\documentclass[11pt]{amsart}

\pdfoutput=1

\usepackage[utf8]{inputenc}
\usepackage[T1]{fontenc}
\usepackage{lmodern}
\usepackage{amsthm, amssymb, amsmath, amsfonts, mathrsfs}

\usepackage{microtype}

\usepackage[pagebackref,colorlinks=true,pdfpagemode=none,urlcolor=blue,
linkcolor=blue,citecolor=blue]{hyperref}

\usepackage{amsmath,amsfonts,amssymb,amsthm}
\usepackage{amsthm, amssymb, amsmath, amsfonts, mathrsfs}

\usepackage{mathrsfs}
\usepackage{MnSymbol}
\usepackage{scalerel} 

\usepackage{color}
\usepackage{accents}

\usepackage{bbm}

\definecolor{labelkey}{gray}{.8}
\definecolor{refkey}{gray}{.8}

\definecolor{darkred}{rgb}{0.9,0.1,0.1}
\definecolor{darkgreen}{rgb}{0,0.5,0}

\newtheorem{theorem}{Theorem}[section]
\newtheorem{lemma}[theorem]{Lemma}
\newtheorem{definition}[theorem]{Definition}

\newtheorem{proposition}[theorem]{Proposition}

\theoremstyle{remark}

\renewenvironment{proof}[1][Proof]{ {\itshape \noindent {#1.}} }{$\Box$
\medskip}

\numberwithin{equation}{section}
\newcommand{\R}{\mathbb{R}}

\newcommand{\Z}{\mathbb{Z}}
\newcommand{\N}{\mathbb{N}}
\newcommand{\Pb}{\mathbb{P}}
\newcommand{\PP}{\mathbf{P}}
\newcommand{\E}{\mathbb{E}}

\newcommand{\F}{\mathcal{F}}

\newcommand{\bbR}{\mathbb{R}}

\newcommand{\G}{\mathcal{G}}

\newcommand{\si}{\sigma}

\newcommand{\EE}{\mathbf{E}}
\newcommand{\1}{\mathbbm{1}}

\newcommand{\cov}{\mathrm{Cov}}

\newcommand{\bP}{\mathbb{P}}

\newcommand{\bT}{\mathbb{T}}

\newcommand{\x}{\mathbf{x}}
\newcommand{\cal}{\mathcal}
\newcommand{\vrho}{\varrho}
\newcommand{\lf}{\lfloor}
\newcommand{\rf}{\rfloor}

\newcommand{\blue}[1]{{\color{black}#1}}

\begin{document}

\title[Winding number of polymer on cylinder]{Fluctuations of the Winding number of a directed polymer on a cylinder}
\author{Yu Gu and Tomasz Komorowski}

\address[Yu Gu]{Department of Mathematics, University of Maryland, College Park, MD 20742, USA}

\address[Tomasz Komorowski]{Institute of Mathematics, Polish Academy
  of Sciences, ul. \'{S}niadeckich 8, 00-656, Warsaw, Poland}

\maketitle

\begin{abstract}
We prove a central limit theorem for the winding number of a directed polymer on a cylinder, which is equivalent with proving the Gaussian fluctuations of the endpoint of the directed polymer in a spatial periodic environment.

\bigskip




\end{abstract}
\maketitle

\section{Introduction}

\subsection{Main result}

We consider the problem of a directed polymer on a cylinder and study
the fluctuations of the winding number, that is, the algebraic number
of turns the polymer path does around the cylinder. The problem is
equivalent \blue{to} studying \blue{the fluctuations of the  endpoint  of a directed polymer in a random  periodic  environment}. To state the main result, we first give an informal description of the model. The reference path measure is chosen to be the Wiener measure, and the random environment is modeled by a Gaussian space-time white noise $\xi(t,x)$ on $\R_+\times [0,1]$, with periodic boundary condition, and we periodically extend it to $\R_+\times \R$. 

For each realization of the random environment, the partition function of the directed polymer is given by 
\begin{equation}\label{e.defZ}
Z_T=\E \exp(\beta \int_0^T\xi(t,w_t)dt),
\end{equation}
where $\{w_t\}_{t\geq0}$ is a one-dimensional standard Brownian motion
starting from the origin, independent of $\xi$, and $\E$ is the
expectation \blue{over the realizations of the Brownian motion}  $w$. Here $\beta>0$ is a fixed parameter playing the role of the inverse temperature. Since $\xi$ is a space-time white noise, the above expression should be interpreted carefully, see Section~\ref{s.she} below for more details.

 The quenched  density of the polymer endpoint $w_T$
 is then given by 
\begin{equation}\label{e.defrho}
\rho(T,x)=\frac{\E \exp(\beta \int_0^T\xi(t,w_t)dt) \delta(w_T-x)}{\E \exp(\beta \int_0^T\xi(t,w_t)dt)}.
\end{equation}
Since the random environment $\xi$ is periodic in space, another perspective is to view the polymer path as lying on a cylinder, in which case it is the trajectory of $\{w_t-\lf w_t\rf\}_{t\geq 0}$ we are tracking. The winding number of the polymer path around the cylinder, denoted by $W_T$, then equals to 
\begin{equation}
\label{WT}
W_T=\lf w_T \rf \1_{w_T\geq0}+\lceil w_T \rceil \1_{w_T<0}.
\end{equation}
Here $\lf\cdot\rf$ and $\lceil \cdot \rceil$ denote  the floor and ceiling
functions, respectively.
Thus, to study the large time behavior of $w_T$ is equivalent to that of $W_T$. 
Denote the quenched probability measure by $\hat{\Pb}_T$ and  the expectation with respect to it by $\hat{\E}_T$, so  for any bounded function $f:\R\to\R$, we have $\hat{\E}_T f(w_T)=\int_\R f(x)\rho(T,x)dx$. Let $\PP, \EE$ be the probability and expectation with respect to the noise $\xi$.  Now we can state the main result of the paper, namely, under the annealed polymer measure $\PP\otimes \hat{\Pb}_T$, $\{\frac{w_T}{\sqrt{T}}\}_{T>0}$, or equivalently, $\{\frac{W_T}{\sqrt{T}}\}_{T>0}$, satisfies a central limit theorem. 
\begin{theorem}\label{t.mainth}
There exists $\sigma_{\mathrm{eff}}^2\in(0,\infty)$, given in \eqref{e.defsigma} below,  such that for any $\theta\in\R$, we have
\[
\EE \hat{\E}_T \exp(i\theta \frac{w_T}{\sqrt{T}})\to \exp(-\frac12\sigma_{\mathrm{eff}}^2 \theta^2), \quad\quad \mbox{ as }T\to\infty.
\]

\end{theorem}


\subsection{Context and motivation}
Our study of the winding number is motivated by the  work of Brunet
\cite{brunet}, where the same problem was investigated by the replica
method. What is particularly interesting  is the exact formula he
derived  for $\sigma_{\mathrm{eff}}^2$ and how it depends on the size
of the period,  see \cite[Eq. (19)-(20)]{brunet}.  It is not hard to
convince oneself that Theorem~\ref{t.mainth} holds for any spatial
period, with the effective diffusion constant depending on the size of
the cell -- we chose the length $L=1$ only to simplify the notations. Denote the corresponding variance by $\sigma_{\mathrm{eff}}^2(L)$. It is known that the polymer path is  super-diffusive  with the
exponent $2/3$ when $L=\infty$, i.e., if there is no periodic structure,  $T^{-2/3}w_T$ is of order $O(1)$
for $T\gg1$, see \cite[Theorem 1.11]{corwin} for relevant results on this particular model. To go from the diffusive   to super-diffusive scaling as $L\to\infty$, it
is natural to expect  $\sigma_{\mathrm{eff}}^2(L)$ to blow up. This was indeed predicted in \cite{brunet}: as
$L\to\infty$, $\sigma_{\mathrm{eff}}^2(L)\sim \sqrt{L}$. The blow up
rate is related to the $2/3$ super-diffusion exponent, and here is a heuristic explanation: for cells of size $L$, the
displacement of the endpoint $w_T$ is of the order
$\sigma_{\mathrm{eff}}(L)\sqrt{T}$, provided that $L\sim O(1)$ and
$T\gg1$. As we keep $T\gg1$ fixed and slowly increase $L$, the polymer path would still visit
many cells provided that $\sigma_{\mathrm{eff}}(L)\sqrt{T} \gg L$. In this case
we still expect to see a homogenization phenomenon and the central limit
theorem as in Theorem~\ref{t.mainth} holds. So, it is natural to guess that the critical scale comes from balancing the two terms,  $\sigma_{\mathrm{eff}}(L)\sqrt{T}$ and $L$. This leads to $L\sim T^{2/3}$, under the assumption of $\sigma_{\mathrm{eff}}^2(L)\sim \sqrt{L}$. 

It was our hope to prove the above heuristics rigorously, and to confirm (or disprove) the replica calculations in \cite{brunet}. Theorem~\ref{t.mainth} can be viewed as a small step towards this goal, through which we confirmed the diffusive scaling and the Gaussian fluctuations. The formula derived for $\sigma_{\mathrm{eff}}^2$, see \eqref{e.defsigma} below, is of Green-Kubo type which involves the integral of some covariance function, and is too implicit  to perform any asymptotic analysis.  This is not surprising though, since the homogenization constant is usually given by the solution to some cell problem and precise estimates on it are   not easy to obtain (see a very recent contribution along this line for a different model of diffusion in random environment \cite{otto}). Guided by the same philosophy, a similar study was carried out for the fluctuations of the free energy $\log Z_T$, which leads to the optimal size of fluctuations in certain regimes where $L,T$ go to infinity together, see \cite{dgk}. 

For the connections between the winding number of the directed polymer in random environment and other models in statistical physics, such as vortices in superconductors and strongly correlated fermions, we refer to \cite{brunet} and the references cited there.

One can also formulate the problem as a diffusion in a
distribution-valued random environment and study the corresponding SDE
with a singular drift, see e.g. \cite{dr,cc,gp} and the references
therein. In this framework, making sense of the singular diffusion is
already highly nontrivial, and is intimately related to the study of
singular SPDE \cite{hairer1,hairer2,gip}.  For our specific problem of
directed polymer, one can view it as a passive scalar with the
velocity field given by the solution of a stochastic Burgers equation,
see \cite[Theorem 31]{dr} which gives a rigorous meaning of
it. Although the velocity field is spatial periodic, which is
sometimes viewed as the simple case in the study of homogenization
or invariance principle of diffusion in a random environment, the
problem does not  fall into any classical framework. It might be
possible to employ the tools developed for singular diffusions and
combine with homogenization type of arguments to study  similar
problems and to prove central limit type results. For this particular
problem, we make use of the structure of the Gibbs measure and give a
proof using a classical argument for   the central limit theorem for
weakly dependent random variables.

\subsection{Sketch of proof} 
Our approach relies heavily on the previous work of studying the
periodic KPZ equation \cite{GK21}, where we showed the endpoint
distribution of the directed polymer on a cylinder mixes exponentially
fast. The proof in \cite{GK21} was inspired by the classical work of
Sinai \cite{sinai}, which was on the stochastic Burgers equation
on the torus. Similar results were also  obtained in \cite{rosati},
using a random version of the Krein-Rutman theorem. As mentioned
previously, one could view the polymer path as lying on the cylinder
by considering the path $\{w_t-\lf w_t\rf\}_{t\geq0}$. Assuming that
at each integer time $k$, the position of the path is $x_k$, i.e., $w_k-\lf
w_k\rf=x_k$, our previous result implies a strong mixing property of
$\{x_k\}_{k\geq 1}$ under the polymer measure. If we denote $\eta_k$
the winding number of the polymer path accumulated during the interval
$[k-1,k]$, then the total winding number is simply $\sum_k\eta_k$. It
is not hard to deduce that, given the positions of $\{x_k\}_{k\geq
  1}$, the sequence of random variables $\{\eta_k\}_{k\geq 1}$ are
independent, so, the correlation only comes from the correlation in
those $x_k$. Our strategy will be to first consider the case when the
starting and ending points $x_0$ and $x_T$ are both sampled from the
stationary measure so that $\{\eta_k\}_{k\geq 1}$ is a sequence of
stationary random variables, and we will prove a   $\rho-$mixing
(correlation mixing) condition to apply the general central limit theorem for the sum of stationary random variables. Then, to finish the proof, we will show that the error induced by resampling $x_0$ and $x_T$ is asymptotically small, again using the strong mixing property of $\{x_k\}_{k\geq 1}$.

The same proof applies verbatim to the high dimensional setting when the random environment is assumed to be white in time and smooth in space.

\subsection*{Organization of the paper} In Section~\ref{s.she}, we
formulate the problem rigorously, define the endpoint distribution
through a stochastic heat equation, and construct a Markov chain which
keeps tracking the winding number of the polymer path as time
increases. Sections~\ref{s.clt} and \ref{s.clt1} are devoted to
proving the main result, by first reducing it to the stationary
setting, then proving the $\rho-$mixing condition for the stationary sequence $\{\eta_k\}$.  In Section~\ref{s.diff}, we prove the nondegeneracy of the variance $\sigma_{\mathrm{eff}}^2$. Some further discussion is left in Section~\ref{s.dis}.

\subsection*{Notations} We will sometimes use the shorthand integral notation $\int$ when the domain of integration is clear from the context. If we do not specify the range of the summation in $\sum_j$, it stands for $\sum_{j\in\Z}$.

\subsection*{Acknowledgement} We thank the anonymous referee for multiple suggestions and comments. Y.G. was partially supported by the NSF through DMS-2203014.  T.K. acknowledges the support of NCN grant 2020/37/B/ST1/00426.

\section{Preparations}
\label{s.she}

\subsection{Stochastic heat equation and endpoint density on $\R$}
As mentioned previously, the expression $Z_T=\E
\exp(\int_0^T\xi(t,w_t)dt)$ is only formal since $\xi$ is a space-time
white noise, and we actually need to consider the so-called Wick
exponential.  In this section, we define the endpoint density $\rho$
rigorously, through the stochastic heat equation (SHE). For an
excellent introduction to the theory of the stochastic heat equation, we refer to the monograph \cite{davar}.

Consider the equation of the form 
\begin{equation}\label{e.she}
\begin{aligned}
&\partial_t u(t,x;\nu)=\frac12\Delta u(t,x;\nu)+\beta \xi(t,x) u(t,x;\nu), \quad\quad  t>0, x\in\R,   \\
 &u(0,dx)=\nu(dx),
\end{aligned}
\end{equation}
where $\beta>0$, the product between $u$ and $\xi$ is interpreted in the Ito-Walsh sense, and $\nu\in{\cal M}_1(\R)$ - the set of Borel probability measures on $\R$. Denote
the propagator of the above equation by $Z_{t,s}(x,y)$, i.e. for each $(s,y)\in\R_+\times\R$ fixed, we have 
\begin{equation}\label{e.eqZ}
\begin{aligned}
&\partial_t Z_{t,s}(x,y)=\frac12\Delta_x Z_{t,s}(x,y)+\beta \xi(t,x)Z_{t,s}(x,y), \quad\quad t>s, x\in\R,\\
&Z_{s,s}(x,y)=\delta_y(x).
\end{aligned}
\end{equation}
Due to the $1$-periodicity of the noise, we obviously have
\begin{equation}
\label{Zts}
 Z_{t,s}(x+j,y+j)=Z_{t,s}(x,y),\quad j\in\Z,\,x,y\in \R, \,t>s.
\end{equation}
Let $\bT=[0,1]$ be the unit torus with the end points identified in the usual way.
Since $\xi$ is periodic, we can consider the same equation on $\bT$
with the periodic boundary condition. Then its propagator is given by 
\begin{equation}
\label{Gts}
G_{t,s}(x,y)=\sum_j Z_{t,s}(x+j,y)=\sum_j Z_{t,s}(x,y-j),  
\quad x,y\in\bT.
\end{equation}
In other words, $G_{t,s}(x,y)$ is the periodic solution to
\eqref{e.eqZ} with the initial data \blue{$G_{s,s}(x,y)=  \sum_{j}\delta_{y-j}(x)$}. 

With the above notations, the random density $\rho$, which is the  density of $w_T$ under the quenched polymer measure $\hat{\Pb}_T$ and was formally defined in \eqref{e.defrho}, takes the form  
\begin{equation}\label{e.defrho1}
\rho(T,x;\nu)=\frac{u(T,x;\nu)}{\int_\R u(T,x';\nu)dx'}.
\end{equation}
From now on, we choose $\nu(dx)$ to be the Dirac measure at the origin. To simplify the notation, we will omit the dependence on $\nu$ when there is no confusion.

\subsection{Endpoint density on the cylinder}
\label{s.polyt}
Besides studying the polymer endpoint on the whole line, we also consider its periodic counterpart:
 $$
\rho_{\rm per}(t,x;\nu)=\frac{v(t,x;\nu)}{\int_{\bT} v(t,x';\nu)dx'},
$$
where $\nu\in\mathcal{M}_1(\bT)$ and $v$ solves the equation 
\begin{equation}\label{e.she1a}
\begin{aligned}
&\partial_t v(t,x;\nu)=\frac12\Delta v(t,x;\nu)+\beta \xi(t,x) v(t,x;\nu), \quad\quad  t>0, x\in\bT,   \\
 &v(0,dx)=\nu(dx).
\end{aligned}
\end{equation}
Using the propagator, the solution can be written as 
\[
v(t,x;\nu)=\int_{\bT} G_{t,0}(x,y)\nu(dy).
\]
It turns out, see \cite[Lemma 2.2]{GK21}, that $ \{\rho_{\rm per}(t
)\}_{t\geq 0}= \{\rho_{\rm per}(t,\cdot;\nu)\}_{t\geq 0,\nu\in{\cal
    M}_1}$ is a \blue{Markov family.   For any $t>0$, the random element
  $\rho_{\rm per}(t,\cdot;\nu) $ takes values in ${\mathbb D}_c(\bT)$, which we use to denote the space
  of continuous probability densities on $\bT$. }

To simplify the notation, for any $t>s$, we define the forward and backward polymer densities starting from $\nu$ by
\begin{equation}\label{020107-22}
\begin{aligned}
&\rho_{\rm per}(t,x;s,\nu)=\frac{\int_{\bT}G_{t,s}(x,y)\nu(dy)}{\int_{\bT^2}
  G_{t,s}(x',y')\nu(dy')dx'},\\
&\tilde{\rho}_{\rm per}(t,\nu;s,x)=\frac{\int_{\bT}G_{t,s}(y,x)\nu(dy)}{\int_{\bT^2} G_{t,s}(y',x')\nu(dy')dx'}.
\end{aligned}
\end{equation}
\blue{We have $\rho_{\rm per}(t,x;\nu) = \rho_{\rm per}(t,x;0,\nu)$.} By the time reversal of the space-time white noise, \blue{for any $t>s$ and
$\nu$ fixed,} we have 
\[
\{\rho_{\rm per}(t,x;s,\nu)\}_{x\in\bT}\stackrel{\text{law}}{=}\{\tilde{\rho}_{\rm per}(t,\nu;s,x)\}_{x\in\bT}
\]
We emphasize that, throughout the paper, the notation $\rho(t,\cdot;\nu)$ is the endpoint density of the polymer on the whole line, while the notations $\rho_{\rm per} (t,\cdot;s,\nu)$ and $\tilde{\rho}_{\rm per} (t,\nu;s,\cdot)$ are reserved for the endpoint density on the torus.

Now we summarize a few results concerning  the
properties  of the Markov family $\{\rho_{\rm per}(t)\}_{t\geq0}$, see
\cite[Theorem 2.3, Eq. (4.17), Lemma 4.1, Proposition 4.6]{GK21}. 

\begin{proposition}
\label{p.tk1}
There exists a unique invariant measure $\pi_\infty$ for  $\{\rho_{\rm per}(t)\}_{t\geq0}$, supported on ${\mathbb D}_c(\bT)$. For any $p\geq 1$, there exists $C,\lambda>0$ such that for all $t>1$, 
\begin{equation}
\EE \sup_{\nu,\nu'\in\mathcal{M}_1(\bT)} \sup_{x\in\bT} |\rho_{\rm per}(t,x;\nu)-\rho_{\rm per}(t,x;\nu')|^p  \leq Ce^{-\lambda t}, 
\end{equation}
and
\begin{equation}\label{e.mmbdrho}
\EE \sup_{\nu\in \mathcal{M}_1(\bT)}\sup_{x\in\bT}\{\rho_{\rm per}(t,x;\nu)^p +\rho_{\rm per}(t,x;\nu)^{-p}\} \leq C.
\end{equation} 
\end{proposition}


\subsection{A Markov chain for the winding number}
\label{s.markov}

 As the random environment is periodic in space, to study the
 displacement of the polymer endpoint $w_T$, it is equivalent to
 studying the winding number of the polymer path when we view it as lying on a cylinder by considering the trajectory $\{w_t-\lf w_t\rf\}_{t\in[0,T]}$. This is the perspective we will take from now on. The idea is to first sample the trajectory of the polymer path on the cylinder at integer times, then consider the winding of the path between successive integer times.


For any $N\in \Z_+$, consider $u(N,j_N+x_N)$ where $x_N\in [0,1)$ and $j_N\in \Z$. By the definition of the propagator, we have 
\[
\begin{aligned}
&u(N,j_N+x_N)=\int_{\R} Z_{N,N-1}(j_N+x_N,y)u(N-1,y)dy\\
&=\sum_{j_{N-1}}\int_{\bT}Z_{N,N-1}(j_N+x_N,j_{N-1}+x_{N-1})u(N-1,j_{N-1}+x_{N-1})dx_{N-1}.
\end{aligned}
\]
Iterate the above relation, we reach at (recall that $u(0,x)=\delta(x)$)
\[
\begin{aligned}
&u(N,j_N+x_N)\\
&= \sum_{j_1,\ldots,j_{N-1}}\int_{\bT^{N-1}} \prod_{k=1}^N Z_{k,k-1}(j_k+x_k,j_{k-1}+x_{k-1})d\x_{1,N-1},
\end{aligned}
\]
where we used the simplified notation $d\x_{1,N-1}=dx_1\ldots dx_{N-1}$ and the convention $
j_0=x_0=0.$
 In other words, in the above integration, we have decomposed the domain $\R$ as $\R=\cup_j [j,j+1)$, then integrate in each interval and sum them up. One should think of the variable $j_k+x_k$ as representing the location of the polymer path at time $k$, with $j_k$ the integer part and $x_k$ the fractional part, i.e., $j_k=\lf w_k\rf$ and $x_k=w_k-\lf w_k\rf$.

Now we make use of the periodicity and observe that 
\begin{equation}\label{e.peZG}
\begin{aligned}
&\sum_{j_k}Z_{k,k-1}(j_k+x_k,j_{k-1}+x_{k-1})=\sum_{j_k} Z_{k,k-1}(j_k-j_{k-1}+x_k,x_{k-1})\\
&=\sum_{j_k} Z_{k,k-1}(j_k+x_k,x_{k-1})=G_{k,k-1}(x_k,x_{k-1}),
\end{aligned}
\end{equation}
where $G$ is the periodic propagator defined in \eqref{Gts}. Then we can write 
\begin{equation}\label{e.7151}
\begin{aligned}
u(N,j_N+x_N)
=\int_{\bT^{N-1}}  &\left(\sum_{j_1,\ldots,j_{N-1}} \frac{\prod_{k=1}^N Z_{k,k-1}(j_k+x_k,j_{k-1}+x_{k-1})  }{\prod_{k=1}^N G_{k,k-1}(x_k,x_{k-1}) }\right)\\
&\times \prod_{k=1}^N G_{k,k-1}(x_k,x_{k-1})  dx_{1,N-1}.
\end{aligned}
\end{equation}

Fix the realization of the random noise and  
\begin{equation}\label{e.bx}
\x=(x_0,x_1,\ldots,x_{N})\in \bT^{N+1}.
\end{equation} 
We construct an integer-valued, time inhomogeneous Markov chain $\{Y_j\}_{j=1}^N$, with 
\begin{equation}\label{e.6174}
\begin{aligned}
& \Pb_\x[Y_1=j_1]= \frac{Z_{1,0}(j_1+x_1,x_0)}{G_{1,0}(x_1,x_0)},\\
&\Pb_\x[Y_2=j_2|Y_1=j_1]= \frac{Z_{2,1}(j_2+x_2,j_1+x_1)}{G_{2,1}(x_2,x_1)},\\
&\ldots\\
&\Pb_\x[Y_N=j_N|Y_{N-1}=j_{N-1}]=\frac{Z_{N,N-1}(j_N+x_N,j_{N-1}+x_{N-1})}{G_{N,N-1}(x_N,x_{N-1})}.
\end{aligned}
\end{equation}
With the Markov chain, one can write the summation in \eqref{e.7151} as 
\[
\sum_{j_1,\ldots,j_{N-1}}\frac{\prod_{k=1}^N Z_{k,k-1}(j_k+x_k,j_{k-1}+x_{k-1})  }{\prod_{k=1}^N G_{k,k-1}(x_k,x_{k-1})  }=\Pb_{\mathbf{x}}[Y_N=j_N],
\]
where, to emphasize the dependence of the Markov chain on
$x_0,x_1,\ldots,x_N$, we have denoted the probability by $\Pb_{\mathbf{x}}$.

In this way, \eqref{e.7151} is rewritten as 
\begin{equation}\label{e.6172}
\begin{aligned}
u(N,j_N+x_N)=\int_{[0,1]^{N-1}}\Pb_{\mathbf{x}}[Y_N=j_N]\prod_{k=1}^N G_{k,k-1}(x_k,x_{k-1})   d\x_{1,N-1}.
\end{aligned}
\end{equation}
Recall that $\hat{\Pb}_N$ is the quenched probability of the polymer
measure on paths of length $N$, and $\lf w_N\rf $ is the integer part
of the endpoint $w_N$. Then 
\begin{equation}\label{e.6173}
\begin{aligned}
&\hat{\Pb}_N[ \lf w_N\rf=j_N]=\int_{\bT} \rho(N,j_N+x_N)dx_N=\frac{\int_{\bT}u(N,j_N+x_N)dx_N}{\int_{\R} u(N,x')dx'}\\
&=\frac{\int_{\bT^{N}}\Pb_{\mathbf{x}}[Y_N=j_N]\prod_{k=1}^N G_{k,k-1}(x_k,x_{k-1})   d\x_{1,N}}{\int_{\bT^{N}}\prod_{k=1}^N G_{k,k-1}(x_k,x_{k-1})   d\x_{1,N}}.
\end{aligned}
\end{equation}
In other words, the quenched distribution of $\lf w_N\rf $ is a weighted average of the distribution of $Y_N$ (the average is over the $\x$ variable). 

We introduce another notation: suppose that $f,g\in{\mathbb D}_c(\bT)$, define
\begin{equation}\label{e.muN3}
\mu_N(\x;f, g ):=\frac{f(x_N)\prod_{k=1}^N
  G_{k,k-1}(x_k,x_{k-1})g(x_0)}{G_{N,0}(f,g)},
\end{equation}
with $\x=(x_0,\ldots,x_N)$ and the normalization factor 
\begin{equation}
\label{GNM}
G_{N,0}(f,g):=\int_{\bT^{N+1}}f(x_N)\prod_{k=1}^N
  G_{k,k-1}(x_k,x_{k-1})  g(x_0)  d\x_{0,N},
\end{equation}
where   $d\x_{0,N}:=dx_0\ldots dx_N$. For each realization of the random environment, one should view $\mu_N(\x;f,g)$ as the joint density of the polymer on the cylinder, evaluated at $(0,x_0),(1,x_1),\ldots,(N,x_N)$, with the starting and ending points sampled from the densities $g,f$ respectively. For any $\nu,\nu'\in \mathcal{M}_1(\bT)$, we abuse the notation and write $\mu_N(\x;\nu,\nu')$ as well, meaning that the starting and ending points are sampled from $\nu',\nu$. In this case, $G_{N,0}(\nu,\nu')$ equals to
\[
G_{N,0}(\nu,\nu'):=\int_{\bT^{N+1}}\prod_{k=1}^N
  G_{k,k-1}(x_k,x_{k-1})  \nu'(dx_0)  d\x_{1,N-1}\nu(dx_N).
  \]

With the above new notation, we can rewrite 
\begin{equation}\label{e.pbn}
\hat{\Pb}_N[\lf w_N\rf =j_N]=\int_{\bT^{N+1}}\Pb_{\mathbf{x}}[Y_N=j_N] \mu_N(\x;{\rm m},\delta_0) d\x_{0,N},
\end{equation}
where ${\rm m}$ is the Lebesgue measure on $\bT$ (note that in \eqref{e.6173}, the convention is $x_0=0$). 


By \eqref{e.peZG} and \eqref{e.6174}, it is clear that $Y_N$ is a sum of independent random variables, for each fixed realization of the noise and $\x$. We rewrite it as 
\begin{equation}\label{e.defeta}
Y_N=\sum_{k=1}^N \eta_k, \quad\quad \eta_k=Y_k-Y_{k-1}, \quad\quad Y_0=0.
\end{equation}
One should interpret $\eta_k$ as the winding number accumulated during the time interval $[k-1,k]$, and   we have 
\begin{equation}\label{e.laweta}
\Pb_\x[\eta_k=j]=\frac{Z_{k,k-1}(x_k+j,x_{k-1})}{G_{k,k-1}(x_k,x_{k-1})}, 
\quad\quad j\in \Z.
\end{equation}

To prove Theorem~\ref{t.mainth} for the winding number $W_N$, see \eqref{WT}, or the endpoint $w_N$, it is equivalent to proving it for $\lf w_N\rf$. From now on, we will focus on the law of $\lf w_N\rf$ and the rest of the analysis starts from  the representation \eqref{e.pbn}.

\section{Proof of the central limit theorem}
\label{s.clt}

%
%
%
The goal is to prove the central limit theorem for
$\frac{w_T}{\sqrt{T}}$ under the annealed polymer measure $\PP\otimes \hat{\Pb}_T$. 
For  $\theta\in\bbR$, define
$$
\varphi_T(\theta):=\EE \hat{\E}_T e^{i\theta w_T/\sqrt{T}}=\EE \int_{\R} \exp\left\{\frac{i\theta  x}{\sqrt{T}}\right\}\rho(T,x)dx,
$$
where
$\rho$ was defined in \eqref{e.defrho} with $\nu$ chosen to be the Dirac measure at the origin. In this section, we will  consider those $T$ taking integer values, and the main goal is to show
\begin{theorem}
  \label{clt}
  We have
  \begin{equation}
  \lim_{N\to\infty}\varphi_N(\theta)=\exp\left\{-\frac{(\si_{\mathrm{eff}}\theta)^2}{2}\right\},\quad
  \theta\in\bbR,
\end{equation}
with   $\si_{\mathrm{eff}}$ given by \eqref{e.defsigma} below.
  \end{theorem}

To show the above theorem it suffices to consider the   integer part of $w_N$. Define 
\begin{align}
\label{psiN1}
\psi_N(\theta)=\EE \int  \exp\left\{\frac{i\theta
  \lf x\rf }{\sqrt{N}}\right\}\rho(N,x)dx.
\end{align}
We focus on finding the limit of $\psi_N(\theta)$. From the construction of the Markov chain in Section~\ref{s.markov} and \eqref{e.pbn}, we have
\begin{align*}
\psi_N(\theta)
&=\EE \int_{\bT^N} \E_{\x}\exp\left\{\frac{i\theta
  Y_N}{\sqrt{N}}\right\} \mu_N(\x;{\rm m},\delta_0)d\x_{0,N} \\
  &=\EE \int_{\bT^N} \E_{\x}\exp\left\{\sum_{k=1}^N\frac{i\theta
  \eta_k}{\sqrt{N}}\right\} \mu_N(\x;{\rm m},\delta_0)d\x_{0,N} \\
&
=\EE \int_{\bT^N} \prod_{k=1}^N\E_{\x}\exp\left\{\frac{i\theta
  \eta_k}{\sqrt{N}}\right\} \mu_N(\x;{\rm m},\delta_0)d\x_{0,N} .
\end{align*}
Here we used $\E_\x$ to denote the expectation with respect to $\Pb_\x$.

\subsection{Estimates of the moments of increments}

%

For any $\nu,\nu'\in \mathcal{M}_1(\bT)$ and $p\geq 1$, define
\begin{equation}\label{010107-22}
E_{N,k}^{(p)}(\nu,\nu'):=\EE \int_{\bT^{N+1}} (\E_\x|\eta_k|^p)
\mu_N(\x;\nu,\nu')d\x_{0,N} .
\end{equation}
Recall that $\eta_k$ was defined in \eqref{e.defeta} and its law is given by \eqref{e.laweta}.
The following result holds.
\begin{lemma}
\label{lm010107-22}
For any $p>1$, we have
\begin{align}
\label{040107-22}
 {\frak E}_p:=\sup_{\nu,\nu'\in{\cal
  M}_1(\bT)}\sup_{N\geq 1}\sup_{k=1,\ldots,N} E_{N,k}^{(p)}(\nu,\nu') <\infty.
\end{align}
\end{lemma}

\begin{proof}
By the definition we have
\begin{equation}\label{e.Enk2}
\begin{aligned}
E_{N,k}^{(p)}(\nu,\nu')&=\EE \int_{\bT^2} dx_{k}dx_{k-1} \sum_{j}
|j|^pZ_{k,k-1}(x_k+j,x_{k-1}) \notag\\
&
\times \frac{     G_{N,k}(\nu;x_k)  G_{k-1,0}(x_{k-1};\nu')}{\int_{\bT^2} G_{N,k}(\nu;y')
  G_{k,k-1}(y',y) G_{k-1,0}(y;\nu') dy'dy}.
  \end{aligned}
  \end{equation}
  Here we used the simplified notation $G_{t,s}(x;\nu)=\int_{\bT}
  G_{t,s}(x,y)\nu(dy)$ and
  $G_{t,s}(\nu;x)=\int_{\bT}G_{t,s}(y,x)\nu(dy)$.  The above
  expression is bounded from above by 
  \begin{equation}
  \begin{aligned}
& \EE\Big\{\left(\inf_{z,z'}G_{k,k-1}(z',z)\right)^{-1} \int_{\bT^2}dx_{k}dx_{k-1}\sum_{j}
|j|^pZ_{k,k-1}(x_k+j,x_{k-1}) \notag\\
&
\times \frac{  G_{N,k}(\nu;x_k)   G_{k-1,0}(x_{k-1};\nu')}{\int_{\bT^2} G_{N,k}(\nu;y')
    G_{k-1,0}(y;\nu') dy'dy} \Big\}\\
&=\int_{\bT^2}\EE\Big\{F_k(x_k,x_{k-1})\tilde
  \rho_{\rm per}(N,\nu;k,x_k ) \rho_{\rm per} (k-1,x_{k-1};\nu')\Big\} dx_{k}dx_{k-1}.
\end{aligned}
\end{equation}
Here 
\begin{align}
F_k(x_k,x_{k-1}):=\left(\inf_{z,z'}G_{k,k-1}(z',z)\right)^{-1}\left(\sum_{j}
|j|^pZ_{k,k-1}(x_k+j,x_{k-1}) \right).
\end{align}
Note that   $F_k$, $\tilde  \rho_{\rm per}(N,\nu;k,x_k )$ and  $\rho_{\rm per} (k-1,x_{k-1};\nu')$ are
  independent. 
Therefore, we have
\begin{align*}
E_{N,k}^{(p)}(\nu,\nu')\le \int_{\bT^2}\EE F_k(x_k,x_{k-1}) \EE \tilde  \rho_{\rm per}(N,\nu;k,x_k )\EE \rho_{\rm per} (k-1,x_{k-1};\nu') dx_{k}dx_{k-1}.
\end{align*}
For any $q>1$, we conclude by the H\"older inequality
\begin{align*}
&
\EE F_k(x_k,x_{k-1}) =\sum_{j }
|j|^p\EE \Big\{\left(\inf_{z,z'}G_{k,k-1}(z',z)\right)^{-1} Z_{k,k-1}(x_k+j,x_{k-1}) \Big\}\\
&
\le \sum_{j }|j|^p\left\{\EE\left(\inf_{z,z'}G_{k,k-1}(z',z)\right)^{-q}\right\}^{1/q}\left\{\EE 
Z_{k,k-1}^{q'} (x_k+j,x_{k-1}) \right\}^{1/q'},
\end{align*}
with $1/q+1/q'=1$. There exists a constant $C_q>0$ depending only on $q$
and such that
\[
\begin{aligned}
&\EE\left(\inf_{z,z'}G_{k,k-1}(z',z)\right)^{-q} \leq C_{q}, \\
&
\sup_{x',x } \left\{\EE 
Z_{k,k-1}^{q'} (x'+j,x) \right\}^{1/q'}\leq C_{q}\exp\left\{-\frac{j^2}{C_{q}}\right\},
\end{aligned}
\]
see \cite[Lemma 4.1]{GK21} and Lemma~\ref{l.mmZ} below. 
Hence
\begin{align*}
{\frak F}:=\sup_{k\ge1}\sup_{x',x }\EE F_k(x',x) <\infty
\end{align*}
and   $E_{N,k}^{(p)}(\nu,\nu')\leq {\frak F}$.
This  completes the proof of the lemma.
\end{proof}

The following lemma, concerning  the moments estimate of the
propagator of SHE, is quite standard. For completeness sake we present
its proof in Section \ref{appA}.
\begin{lemma}\label{l.mmZ}
For any $p\geq1$, there exists $C_p>0$ such that
\[
\EE Z_{t,0}(x,0)^p \leq \frac{C_p}{t^{p/2}} \exp\left\{-\frac{x^2}{C_pt}\right\}, \quad\quad \mbox{ for all }t\in(0,2], x\in\R. 
\]
\end{lemma}


\subsection{Characteristic function at equilibrium}
\label{s.cha}

Recall from Section~\ref{s.polyt} that, there exists a unique probability measure $\pi_\infty$ on the space
${\mathbb D}_c(\bT)$ - continuous probability densities on $\bT$ - that is invariant under
the dynamics of the polymer endpoint process.

  Suppose that $\varrho$ and $\tilde \varrho$ are two independent copies of
${\mathbb D}_c(\bT)$-valued random fields, distributed according to
$\pi_\infty$. They are also assumed to be independent of the noise
$\xi$. We will use $\E_\varrho,\E_{\tilde{\varrho}}$ to denote the
expectation with respect to them respectively. From \cite[Theorem
2.3]{GK21} and \eqref{e.mmbdrho}, we know that \eqref{e.mmbdrho} also
holds for $\varrho$, i.e. for any $p\geq 1$ we have 
\begin{equation}\label{e.mmbdrho1}
{\frak R}_p:=\E_\varrho \sup_x \{\varrho(x)^p+\varrho(x)^{-p}\} <+\infty.
\end{equation}
 
Define
\begin{align}
\label{psiN}
\tilde\psi_N(\theta)
=\E_{\varrho}\E_{\tilde \varrho} \EE \int_{\bT^{N+1}} \E_{\x}\exp\left\{\sum_{k=1}^N\frac{i\theta
  \eta_k}{\sqrt{N}}\right\} \mu_N(\x;\tilde\vrho,\vrho) d\x_{0,N},\notag
\end{align}
and recall that 
\[
\psi_N(\theta)=\EE \int_{\bT^{N+1}} \E_{\x}\exp\left\{\sum_{k=1}^N\frac{i\theta
  \eta_k}{\sqrt{N}}\right\} \mu_N(\x;{\rm m},\delta_0)d\x_{0,N}.
\]
The only difference between $\psi_N$ and $\tilde{\psi}_N$ comes from the distributions of the starting and ending points of the directed polymer on the cylinder: for $\psi_N$, the starting point is the origin and the ending point is ``free'' and distributed according to the Lebesgue measure on $\bT$, while for $\tilde{\psi}_N$, the starting and ending points are sampled independently from the stationary distribution.

The purpose of the present section is to show the following proposition, which reduces the proof of central limit theorem to the stationary setting. 
\begin{proposition}
\label{prop010107-22}
For any $\theta\in\R$ we have
\begin{equation}
\lim_{N\to\infty}[\tilde\psi_N(\theta)- \psi_N(\theta)]=0.
\end{equation}
\end{proposition}
\begin{proof}
Consider a sequence $\{k_N\}_N$ such that    $k_N\to\infty$ and $\frac{k_N}{\sqrt{N}}\to0$. Define
\begin{align}
\label{psiN0}
&
 \psi_{N,o}(\theta):= \EE \int_{\bT^{N+1}} \E_{\x}\exp\left\{\sum_{k=k_N}^{N-k_N}\frac{i\theta
  \eta_k}{\sqrt{N}}\right\} \mu_N(\x;{\rm m},\delta_0)d\x_{0,N} ,\\
&
\tilde \psi_{N,o}(\theta):= \E_{\varrho}\E_{\tilde \varrho}\EE \int_{\bT^{N+1}} \E_{\x}\exp\left\{\sum_{k=k_N}^{N-k_N}\frac{i\theta
  \eta_k}{\sqrt{N}}\right\} \mu_N(\x;\tilde\vrho,\vrho) d\x_{0,N}  .\notag
\end{align}
We have
\begin{align*}
&
 \Big|\psi_{N,o}(\theta)-\psi_{N}(\theta)\Big|\\
&
\le \EE \int_{\bT^{N+1}} \E_{\x}\Big|\exp\left\{\sum_{k=k_N}^{N-k_N}\frac{i\theta
  \eta_k}{\sqrt{N}}\right\} -\exp\left\{\sum_{k=1}^{N}\frac{i\theta
  \eta_k}{\sqrt{N}}\right\} \Big|\mu_N(\x;{\rm m},\delta_0)d\x_{0,N}  \\
&
\le   \EE \int_{\bT^{N+1}} \left(\sum_{k=1}^{k_N-1}+\sum_{k=N-k_N+1}^N\right)\E_\x\Big|\frac{\theta
  \eta_k}{\sqrt{N}}   \Big|\mu_N(\x;{\rm m},\delta_0)d\x_{0,N}  \\
&
 \le
  \frac{2k_N|\theta|\sqrt{{\frak E}_2}}{\sqrt{N}}\to0,
\end{align*}
as $N\to\infty$. In the last step, we have applied Lemma~\ref{lm010107-22} with ${\frak E}_2$ defined in \eqref{040107-22}.  By the same proof, we have $|\tilde \psi_{N,o}(\theta)-\tilde\psi_{N}(\theta)|\to0$. Thus, to prove the proposition, it suffices to show that
\begin{align}
\label{010407-22}
 \Big|\tilde \psi_{N,o}(\theta)- \psi_{N,o}(\theta)\Big|\to0.
\end{align}

Recall that for any $t>s$ and $\nu\in \mathcal{M}_1(\bT)$, $\rho_{\rm
  per}(t,\cdot;s,\nu)$ and $\tilde{\rho}_{\rm per}(t,\nu;s,\cdot)$
were defined in \eqref{020107-22}. \blue{We also have $\rho_{\rm
  per}(t,\cdot;\nu)=\rho_{\rm
  per}(t,\cdot;0,\nu)$}.
In $\mu_N(\x;{\rm m},\delta_0)$, we integrate out the variables 
\[
x_0,\ldots,x_{k_N-2},x_{N-k_N+1},\ldots,x_N
\] to obtain
\[
\begin{aligned}
&\int\mu_N(\x;{\rm m},\delta_0) d\x_{0,k_N-2}d\x_{N-k_N+1,N}\\
&=\frac{\tilde{\rho}_{\rm per}(N,{\rm m}; N-k_N ,x_{N-k_N})\left(\prod_{k=k_N}^{N-k_N}
  G_{k,k-1}(x_k,x_{k-1})\right)\rho_{\rm per}(k_{N}-1  ,x_{k_N-1};\delta_0)}{\G_{N-k_N,k_N}({\rm m},\delta_0)},
\end{aligned}
\]
with 
\begin{equation}\label{e.defGG}
\G_{m_2,m_1}(\nu,\nu'):=\int \tilde{\rho}_{\rm per}(N,\nu;m_2,
x)G_{m_2,m_1-1}(x,y)\rho_{\rm per}(m_1-1 ,y;\nu')dxdy.
\end{equation}
This leads to the following expression
\begin{equation}\label{e.7181}
\begin{aligned}
  \psi_{N,o}(\theta)
=\EE \int&\tilde{\rho}_{\rm per}(N,{\rm m}; N-k_N ,x_{N-k_N})\frac{\prod_{k=k_N}^{N-k_N}
  G_{k,k-1}(x_k,x_{k-1}) }{\G_{N-k_N,k_N}({\rm m},\delta_0)}\\
  &\times \rho_{\rm per}(k_{N}-1  ,x_{k_N-1};\delta_0)\E_{\x}\exp\left\{\sum_{k=k_N}^{N-k_N}\frac{i\theta
  \eta_k}{\sqrt{N}}\right\}  d\x_{k_N-1,N-k_N}.
\end{aligned}
\end{equation}
For brevity sake we write $d\x_{m,M}=dx_m\ldots dx_M$ for any $m\le M$.
With the above notations, we can also write $\tilde{\psi}_{N,o}$ as 
\begin{equation}\label{e.7182}
\begin{aligned}
  \tilde{\psi}_{N,o}(\theta)
=\E_{\varrho}\E_{\tilde{\varrho}}\EE \int&\tilde{\rho}_{\rm per}(N,\tilde{\varrho}; N-k_N ,x_{N-k_N})\frac{\prod_{k=k_N}^{N-k_N}
  G_{k,k-1}(x_k,x_{k-1})}{\G_{N-k_N,k_N}(\tilde{\varrho},\varrho)}\\
  &\times \rho_{\rm per}(k_{N}-1 ,x_{k_N-1};\varrho)\E_{\x}\exp\left\{\sum_{k=k_N}^{N-k_N}\frac{i\theta
  \eta_k}{\sqrt{N}}\right\}  d\x_{k_N-1,N-k_N}.
\end{aligned}
\end{equation}

The idea is that, since $k_N\gg1$, we expect $\rho_{\rm
  per}(k_{N-1},\cdot;\delta_0)$ and $\tilde{\rho}_{\rm per}(N,{\rm m}; N-k_N,\cdot)$ to be close to the stationary distribution, so that in the expression of $\psi_{N,o}$, we can    first ``replace'' $\delta_0$ by $\varrho$, then ``replace'' ${\rm m}$ by $\tilde{\varrho}$. In this way, we reach at the expression of $\tilde{\psi}_{N,o}$

We define an intermediate version:
\begin{equation}\label{e.7183}
\begin{aligned}
  \psi_{N,o}^{(1)}(\theta)
=\E_{\varrho}\EE \int&\tilde{\rho}_{\rm per}(N,{\rm m}; N-k_N, x_{N-k_N})\frac{\prod_{k=k_N}^{N-k_N}
  G_{k,k-1}(x_k,x_{k-1}) }{\G_{N-k_N,k_N}({\rm m},\varrho)}\\
  &\times \rho_{\rm per}(k_{N}-1 ,x_{k_N-1};\varrho)\E_{\x}\exp\left\{\sum_{k=k_N}^{N-k_N}\frac{i\theta
  \eta_k}{\sqrt{N}}\right\} d\x_{k_N-1,N-k_N}.
\end{aligned}
\end{equation}
Compare \eqref{e.7183} with \eqref{e.7181}, the only difference comes from replacing $\delta_0$ with the stationary measure $\varrho$. Next we write 
\[
\psi_{N,o}(\theta)- \psi^{(1)}_{N,o}(\theta)={\rm I}_N+{\rm II}_N,
\] where
\begin{equation}
\begin{aligned}
{\rm I}_N
=\E_\varrho\EE \int&\tilde{\rho}_{\rm per}(N,{\rm m}; N-k_N, x_{N-k_N})\frac{\prod_{k=k_N}^{N-k_N}
  G_{k,k-1}(x_k,x_{k-1})}{\G_{N-k_N,k_N}({\rm m},\delta_0)}\\
  &\times \mathcal{E}_1(k_N-1,x_{k_N-1};\delta_0,\varrho)\E_{\x}\exp\left\{\sum_{k=k_N}^{N-k_N}\frac{i\theta
  \eta_k}{\sqrt{N}}\right\}  d\x_{k_N-1,N-k_N},
\end{aligned}
\end{equation}
with 
\begin{equation}\label{e.defE1}
\mathcal{E}_1(k_N-1,\cdot;\delta_0,\varrho):=\rho_{\rm per}(k_N-1,\cdot;\delta_0)-\rho_{\rm per}(k_N-1,\cdot; \varrho),
\end{equation}
and 
\begin{equation}
\begin{aligned}
{\rm II}_N
=\E_{\varrho}\EE \int&\tilde{\rho}_{\rm per}(N,{\rm m}; N-k_N, x_{N-k_N})\left(\prod_{k=k_N}^{N-k_N}
  G_{k,k-1}(x_k,x_{k-1})\right)\rho_{\rm per}(k_{N}-1 ,x_{k_N-1};\varrho)\\
  &\times \mathcal{E}_2(k_N-1;\delta_0,\varrho) \E_{\x}\exp\left\{\sum_{k=k_N}^{N-k_N}\frac{i\theta
  \eta_k}{\sqrt{N}}\right\} d\x_{k_N-1,N-k_N},
\end{aligned}
\end{equation}
with 
\[
\mathcal{E}_2(k_N-1;\delta_0,\varrho):= \G_{N-k_N,k_N}({\rm m},\delta_0)^{-1}-  \G_{N-k_N,k_N}({\rm m},\varrho)^{-1}.
\]
In the following, we will estimate ${\rm I}_N$ and ${\rm II}_N$ separately. 

\emph{Estimates on ${\rm I}_N$.} First we note that the term $| \E_{\x}\exp\left\{\sum_{k=k_N}^{N-k_N}\frac{i\theta
  \eta_k}{\sqrt{N}}\right\} |\leq 1$. Secondly, by the definition of $\G$ in \eqref{e.defGG}, it is straightforward to check that
\[
\begin{aligned}
&\frac{\tilde{\rho}_{\rm per}(N,{\rm m}; N-k_N, x_{N-k_N}) }{\G_{N-k_N,k_N}({\rm m},\delta_0)} \\
  &\leq \frac{\sup_x \tilde{\rho}_{\rm per}(N,{\rm m}; N-k_N,
    x)}{\inf_x \tilde{\rho}_{\rm per}(N,{\rm m}; N-k_N, x)\inf_x
    \rho_{\rm per}(k_N-1 ,x;\delta_0)} \times  \frac{1}{G_{N-k_N,k_N-1}({\rm m},{\rm m})}.
    \end{aligned}
  \]
  Here we used the simplified notation $G_{t,s}(\nu,\nu')=\int G_{t,s}(x,y)\nu(dx)\nu'(dy)$.
Thus, we have 
\[
\begin{aligned}
|{\rm I}_N |\leq &\E_\varrho\EE \left[\frac{\sup_x \tilde{\rho}_{\rm per}(N,{\rm m}; N-k_N,
    x)}{\inf_x \tilde{\rho}_{\rm per}(N,{\rm m}; N-k_N, x)\inf_x
    \rho_{\rm per}(k_N-1 ,x;\delta_0)} \sup_x |\mathcal{E}_1(k_N-1,x;\delta_0,\varrho)|\right]\\
\leq & \sqrt{\E_\varrho\EE \left(\frac{\sup_x \tilde{\rho}_{\rm per}(N,{\rm m}; N-k_N,
    x)}{\inf_x \tilde{\rho}_{\rm per}(N,{\rm m}; N-k_N, x)\inf_x
    \rho_{\rm per}(k_N-1 ,x;\delta_0)} \right)^2 }\\
&\times\sqrt{\E_\varrho\EE \sup_x |\mathcal{E}_1(k_N-1,x;\delta_0,\varrho)|^2}.
\end{aligned}
\]
According to Proposition~\ref{p.tk1}, there exist constants $C,\lambda>0$ such that the above expression is bounded by 
\[
|{\rm I}_N|\leq Ce^{-\lambda k_N}.
\]

\emph{Estimates on ${\rm II}_N$.} The proof is similar to that of ${\rm I}_N$. By \eqref{e.defGG} and the definition of $\mathcal{E}_1$ in \eqref{e.defE1}, we have 
\[
\begin{aligned}
&\mathcal{E}_2(k_N-1;\delta_0,\varrho)\\
&=-\int \tilde{\rho}_{\rm per}(N,{\rm m}; N-k_N, x_{N-k_N}) \frac{
  G_{N-k_N,k_N-1}(x_{N-k_N},x_{k_N-1})}{\G_{N-k_N,k_N}({\rm
    m},\delta_0)\G_{N-k_N,k_N}({\rm m},\varrho)}\\
&
\times \mathcal{E}_1(k_N-1,x_{k_N-1};\delta_0,\varrho)dx_{N-k_N}dx_{k_N-1},
\end{aligned}
\]
which implies that 
\[
\begin{aligned}
  &|\mathcal{E}_2(k_N-1;\delta_0,\varrho)| \\
  &
  \leq \frac{\sup_x
  \tilde{\rho}_{\rm per}(N,{\rm m}; N-k_N, x)
  \sup_x|\mathcal{E}_1(k_N-1,x;\delta_0,\varrho)| }{(\inf_x
  \tilde{\rho}_{\rm per}(N,{\rm m}; N-k_N, x) )^2\inf_x \rho_{\rm per}(k_N-1, x;\delta_0)\inf_x \rho_{\rm per}(k_N-1, x;\varrho)}\\
&\times \frac{1}{ G_{N-k_N,k_N-1}({\rm m},{\rm m})}.
\end{aligned}
\]
This leads to
\begin{align*}
  &|{\rm II}_N|\\
  &
    \leq\E_{\varrho}\EE\frac{(\sup_x \tilde{\rho}_{\rm
    per}(N,{\rm m}, N-k_N ,x))^2\sup_x\rho_{\rm per}(k_{N}-1 ,x;\varrho) \sup_x|\mathcal{E}_1(k_N-1,x;\delta_0,\varrho)| }{(\inf_x \tilde{\rho}_{\rm
    per}(N,{\rm m}, N-k_N ,x))^2\inf_x \rho_{\rm per}(k_N-1
  ,x;\delta_0)\inf_x \rho_{\rm per}(k_N-1 ,x;\varrho)}.
\end{align*}
Applying H\"older inequality as before, we also obtain that  
\[
|{\rm II}_N|\leq Ce^{-\lambda k_N},\quad N=1,2,\ldots.
\]
Thus, we have 
\[
|\psi_{N,o}(\theta)- \psi^{(1)}_{N,o}(\theta)|\leq Ce^{-\lambda k_N}\to0
\]
as $N\to\infty$.

It remains to show that $\tilde{\psi}_{N,o}(\theta)- \psi^{(1)}_{N,o}(\theta)\to0$ as $N\to\infty$. Comparing \eqref{e.7182} and \eqref{e.7183}, the only difference is  ${\rm m}$ being replaced by $\tilde{\varrho}$. By following the same proof for $\psi_{N,o}(\theta)- \psi^{(1)}_{N,o}(\theta)$ verbatim, we conclude the proof of the proposition.
\end{proof}

\subsection{Construction of the path measure}
\label{sec4.3}

Recall that by the construction of the Markov chain in
Section~\ref{s.markov}, the study of the winding number $W_N$, see \eqref{WT}, or equivalently $\lf w_N\rf$,  reduces to that of $Y_N=\sum_{k=1}^N \eta_k$:
\[
\hat{\Pb}_N[\lf w_N \rf=j]=\int_{\bT^{N+1}}\Pb_{\mathbf{x}}[Y_N=j] \mu_N(\x;{\rm m},\delta_0) d\x_{0,N}, \quad\quad \mbox{ for all } j\in \Z.
\]
By the result in Section~\ref{s.cha}, to study the law of $\lf w_N \rf $, we can further replace $\mu_N(\x;{\rm m},\delta_0)$ in the above expression by the stationary density $\mu_N(\x;\tilde{\varrho},\varrho)$. In this section, we construct a path measure to realize $\{\eta_k\}_{k\in \Z}$ as a sequence of stationary random variables, and the proof of the central limit theorem for $\lf w_N\rf $ reduces to that of $\sum_{k=1}^N \eta_k$.

The space $\Z^\Z$ consists of all functions $\sigma:\Z\to\Z$.
For any $k\in\Z$, we denote by $\eta_k: \Z^\Z\to\Z$ the $k$-th coordinate
map, i.e. $\eta_k(\sigma):=\sigma(k)$. Recall that
for any $f,g\in \mathbb{D}_c(\bT)$, we have defined 
\begin{equation*} 
\mu_N(\x;f,g )=\frac{f(x_N)\prod_{k=1}^N
  G_{k,k-1}(x_k,x_{k-1})g(x_0)}{\int_{\bT^{N+1}}f(x_N')\prod_{k=1}^N
  G_{k,k-1}(x_k',x_{k-1}')g(x_0')   d\x_{0,N}'}.
\end{equation*}
In the following, we construct a probability measure ${\cal P}$ on $\Z^\Z$   such that 
\begin{itemize}
\item[1)] for each
$N\ge1$ and $j_1,\ldots,j_N\in\Z$, we have
\begin{align}
\label{010607-22}
&{\cal P}\Big[\eta_1=j_1,\ldots,\eta_N=j_N\Big]\\
&
:=\E_{\vrho}\E_{\tilde\vrho}\EE \int_{\bT^{N+1}}  \bP_{\x}\Big[\eta_1=j_1,\ldots,\eta_N=j_N\Big]
   \mu_N(\x;\tilde\vrho,\vrho) d\x_{0,N} 
  \notag\\
&
=\E_{\vrho}\E_{\tilde\vrho}\EE \int_{\bT^{N+1}}   \prod_{k=1}^N\frac{Z_{k,k-1}(j_k+x_k,x_{k-1})}{G_{k,k-1}(x_k,x_{k-1})}
   \mu_N(\x;\tilde\vrho,\vrho) d\x_{0,N}
             \notag
\end{align}
\item[2)] for any $\ell\in\Z$, we have
\begin{align}
\label{020607-22}
{\cal P}\Big[\eta_1=j_1,\ldots,\eta_N=j_N\Big]={\cal P}\Big[\eta_{\ell+1}=j_1,\ldots,\eta_{\ell+N}=j_N\Big].
\end{align}
\end{itemize}
This is done as follows.  First, we define the family of measures $({\cal P}_N)_{N\ge1}$ on
$\Z^N$ by \eqref{010607-22}. They induce a finite additive set function ${\cal P}$ on the
algebra ${\cal C}$ of cylindrical subsets of   $\Z^\Z$.  To
show that ${\cal P}$ extends to the $\si$-algebra $\si({\cal C})$, it
suffices to prove the following consistency condition: for each
$N\ge1$, \blue{$\ell\ge1$} and $j_1,\ldots,j_N\in\Z$,
\begin{align}
\label{cons}
&
              \blue{ {\cal P}\Big[\eta_1=j_1,\ldots,\eta_N=j_N\Big]}\\
  &
    =\E_{\vrho}\E_{\tilde\vrho}\EE \int  \bP_{\x}\Big[\eta_1=j_1,\ldots,\eta_N=j_N\Big]
   \mu_N(\x;\tilde\vrho,\vrho) d\x_{0,N} \notag\\
&
=\E_{\vrho}\E_{\tilde\vrho}\EE \int 
  \sum_{j_{N+1},\ldots,j_{N+\ell}}\bP_{\x}\Big[\eta_1=j_1,\ldots,\eta_N=j_N,
  \eta_{N+1}=j_{N+1},\ldots, \eta_{N+\ell}=j_{N+\ell}\Big]\notag\\
&
\times 
                                                                     \mu_{N+\ell}(\x;\tilde\vrho,\vrho) d\x_{0,N+\ell} \notag\\
  &
  \blue{=\sum_{j_{N+1},\ldots,j_{N+\ell}} {\cal P}\Big[\eta_1=j_1,\ldots,\eta_N=j_N,
  \eta_{N+1}=j_{N+1},\ldots, \eta_{N+\ell}=j_{N+\ell}\Big]}. \notag
\end{align}
The right hand side of \eqref{cons} equals
\begin{align*}
&\E_{\vrho}\E_{\tilde\vrho}\EE \int 
  \sum_{j_{N+1},\ldots,j_{N+\ell}}\prod_{k=1}^{N+\ell}\frac{Z_{k,k-1}(j_k+x_k,x_{k-1})}{G_{k,k-1}(x_k,x_{k-1})}\notag\\
&
\times 
   \frac{\prod_{k=1}^{N+\ell}
  G_{k,k-1}(x_k,x_{k-1})  \vrho(x_0)\tilde\vrho(x_{N+\ell})}{\int\prod_{k=1}^{N+\ell}
  G_{k,k-1}(x_k',x_{k-1}')\vrho(x_0')\tilde\vrho(x_{N+\ell}') d\x_{0,N+\ell}'}d\x_{0,N+\ell}
\end{align*}
\begin{align*}
& =\E_{\vrho}\E_{\tilde\vrho}\EE \int d\x_{0,N}
   \prod_{k=1}^{N}\frac{Z_{k,k-1}(j_k+x_k,x_{k-1})}{G_{k,k-1}(x_k,x_{k-1})}\notag\\
&
\times 
   \frac{\tilde \rho_{\rm per}(  N+\ell,\tilde\vrho;N,x_N)\left(\prod_{k=1}^{N}
  G_{k,k-1}(x_k,x_{k-1})  \vrho(x_0) \right) }{\int \prod_{k=1}^{N}
  G_{k,k-1}(x_k',x_{k-1}')\vrho(x_0') \tilde \rho_{\rm per}(  N+\ell, \tilde\vrho;N,x_N') d\x_{0,N}'}.
\end{align*}
Here $\tilde \rho_{\rm per}(N+\ell,\tilde\vrho;N,\cdot) $ is the reverse time,
polymer endpoint process
starting at stationarity, see the definition of $\tilde{\rho}_{\rm per}$ in
\eqref{020107-22}. So we know that $\tilde{\rho}_{\rm per}(N+\ell,\tilde\vrho;N,\cdot)$ has the same law as $\tilde{\varrho}$ and is independent of $\varrho$ and the random environment in the interval $[0,N]$.  We can therefore write that the right
hand side of \eqref{cons} equals
\begin{align*}
&\E_{\vrho}\E_{\tilde\vrho}\EE \int d\x_{0,N}
   \prod_{k=1}^{N}\frac{Z_{k,k-1}(j_k+x_k,x_{k-1})}{G_{k,k-1}(x_k,x_{k-1})}\notag\\
&
\times 
   \frac{ \prod_{k=1}^{N}
  G_{k,k-1}(x_k,x_{k-1})\tilde\vrho(x_N) \vrho(x_0)   }{\int \prod_{k=1}^{N}
  G_{k,k-1}(x_k',x_{k-1}')\vrho(x_0') \tilde\vrho(x_N')  d\x_{0,N}'},
\end{align*}
which proves  \eqref{cons}. The argument for  \eqref{020607-22} 
is similar (for $\ell\geq 1$). This way we construct a stationary measure ${\cal P}$ on
$\Z^{\N}$. Its extension to $\Z^\Z$ is standard.

From now on, we use $E$ to denote the expectation with respect to ${\cal P}$.
  \begin{proposition}
    \label{prop011207-22}
    We have
    \begin{align}
      &E \eta_j=0, \label{021207-22}\\
      &E \eta_j^2<\infty, \qquad \mbox{for all $j\in\Z$.}\label{021207-22a}
    \end{align} 
    \end{proposition}
    
    \begin{proof}
   First, \eqref{021207-22a} is a direct consequence of Lemma~\ref{lm010107-22}. Now we prove \eqref{021207-22}. Define \begin{align}
&
 M_N:= \EE \int \E_{\x}\frac{\sum_{k=1}^N \eta_k}{N} \mu_N(\x;{\rm m},\delta_0)d\x_{0,N} =\EE \int  \E_{\x}\frac{Y_N}{N} \mu_N(\x;{\rm m},\delta_0)d\x_{0,N} ,\\
&
\tilde{M}_N:= \E_{\varrho}\E_{\tilde \varrho}\EE \int  \E_{\x}\frac{\sum_{k=1}^N \eta_k}{N} \mu_N(\x;\tilde\vrho,\vrho) d\x_{0,N} =E \eta_1 .\notag
\end{align}
    By following the proof of Proposition~\ref{prop010107-22} and applying Lemma~\ref{lm010107-22}, we have 
    $M_N-\tilde{M}_N\to0$ as $N\to\infty$. On the other hand, note that $M_N=N^{-1}\EE \hat{\E}_N \lf w_N\rf $. By symmetry we have $\EE \hat{\E}_N w_N=0$, which implies that $M_N\to0$ as $N\to\infty$, since $w_N-1<\lf w_N\rf \leq w_N$. This further implies that $\tilde{M}_N=E \eta_1=0$. The proof is complete.
    \end{proof}

\subsection{The  correlation   mixing}

\label{sec3.4}
In this section, we will show that the stationary sequence $\{\eta_k\}_{k\in\Z}$ constructed in Section~\ref{sec4.3} satisfies a central limit theorem:
\begin{equation}\label{e.clt1}
\frac{\sum_{k=1}^N \eta_k}{\sqrt{N}}\Rightarrow N(0,\sigma_{\mathrm{eff}}^2), \quad\quad \mbox{ as } N\to\infty.
\end{equation}
With \eqref{e.clt1}, Proposition~\ref{prop010107-22}, we conclude the proof of Theorem~\ref{clt}.

Consider the probability space space $(\Z^\Z,\si({\cal C}),{\cal
  P})$. Let ${\cal F}_j$ and ${\cal F}^j$ be the $\si$-algebras generated by
$\{\eta_k\}_{k\le j}$ and $\{\eta_k\}_{k\ge j}$, respectively. Define 
 the correlation coefficient between two square integrable
   random variables $X$, $Y$   as
   $$
{\rm corr}[X,Y]=\frac{\cov[X,Y]}{(E X^2)^{1/2}(E Y^2)^{1/2}}.
   $$ We have the following definition of the $\rho-$mixing coefficient:
 \begin{definition}($\rho$-mixing coefficients, see \cite[Section 19]{bil})
   \label{df-cor}
   The  $\rho$-mixing coefficients for the stationary sequence $\{\eta_k\}_{k\in\Z} $ are defined as
   \begin{equation}
     \label{rho-n}
     r(n):=\sup\left\{|{\rm corr}[F,G]|:\,F\in {\cal F}_j,\,G\in {\cal
         F}^{j+n}\right\},\quad n=1,2,\ldots.
     \end{equation}
   \end{definition}   
Note that, due to the stationarity, the definition of $ r(n)$ in \eqref{rho-n} does
not depend on $j$. The main result of this section is the following: 
\begin{proposition}
  \label{thm011207-22}
There exist
$C,\lambda >0$ such that 
\begin{equation}
\label{021307-22b}
|\cov[F,G]|\leq Ce^{-\lambda n}\|F\|_{L^2}\|G\|_{L^2}
\end{equation}
for any $n\geq 1$ and $F,G$ that are $\F_j$ and $\F^{j+n}$ measurable 
respectively. As a consequence, we have 
  \begin{equation}
     \label{rho-n1}
     r(n)\le Ce^{-\lambda n}.
     \end{equation}
\end{proposition}

%
From \eqref{rho-n1}, Proposition \ref{prop011207-22}
   and \cite[Theorem 19.2]{bil}, we immediately conclude the proof of \eqref{e.clt1}, with 
   \begin{equation}\label{e.defsigma}
   \sigma_{\mathrm{eff}}^2=\sum_{j\in\Z} E[\eta_0\eta_j].
   \end{equation}

Note that $\sigma_{\mathrm{eff}}^2<\infty$ is a direct consequence of \eqref{rho-n1} and \eqref{021207-22a}. We will show $\sigma_{\mathrm{eff}}^2>0$ in Section~\ref{s.diff} below.

The rest of the section is devoted to the proof of Proposition~\ref{thm011207-22}. 

\begin{proof}
Recall that $ E$ denote the expectation with respect to ${\cal
  P}$. 
  Suppose   that 
\begin{align*}
F=f\big(\eta_{1},\ldots,\eta_{m_1}\big),\quad G=g\big(\eta_{m_1+n},\ldots,\eta_{m_2+n}\big)
\end{align*}
for some $m_1,m_2\in \N$ and Borel measurable functions 
$$
f:\bbR^{m_1}\to\bbR,\quad g:\bbR^{m_2-m_1+1}\to\bbR.
$$
Throughout the proof, to simplify the notation, define 
\[
\x=(x_0,\ldots,x_N), \quad\quad 
N=m_2+n.
\]
We have 
\begin{equation}\label{e.covFG}
\begin{aligned}
 E[FG]=\E_\varrho \EE \int  \E_\x[f\big(\eta_{1},\ldots,\eta_{m_1}\big)g\big(\eta_{m_1+n},\ldots,\eta_{m_2+n}\big)]\mu_{N}(\x;\varrho_1,\varrho_2)d\x_{0,N},
\end{aligned}
\end{equation}
where $\varrho_1,\varrho_2$ are sampled independently from $\pi_\infty$, also independent from $\xi$, and  $\E_\varrho$ is the expectation on them. We recall that 
\[
\mu_N(\x;\varrho_1, \varrho_2):=\varrho_1(x_N)\frac{\prod_{k=1}^N
  G_{k,k-1}(x_k,x_{k-1})}{G_{N,0}(\varrho_1,\varrho_2)}\varrho_2(x_0)
  \]
  is a density function, 
  and 
  \[
  G_{N,0}(\varrho_1,\varrho_2)=\int \varrho_1(x_N)\prod_{k=1}^N
  G_{k,k-1}(x_k,x_{k-1})  \varrho_2(x_0) d\x_{0,N}.
  \]
  The goal is to show that, when $N$ is large, the density $\mu_N(\x;\varrho_1,\varrho_2)$ factorizes into two independent ones, with an error that is exponentially small in $n$. The   proof consists of several steps.  
  
  \emph{Step 1. Rewriting $E[FG]$.}
  Since
  $\E_\x[f\big(\eta_{1},\ldots,\eta_{m_1}\big)g\big(\eta_{m_1+n},\ldots,\eta_{m_2+n}\big)]$
  only depends on the variables $\x_{0,m_1},\x_{m_1+n-1,N}$, we will
  first integrate out other variables in
  $\mu_N(\cdot;\varrho_1,\varrho_2)$. We keep a ``middle'' one for a future purpose: define 
\[
\ell=m_1+\lf n/2\rf.
\]
After integrating out the variables $\x_{m_1+1,\ell-1},\x_{\ell+1,m_1+n-2}$, we obtain  
\begin{equation}\label{e.7191}
\begin{aligned}
\int& \mu_N(\x; \varrho_1, \varrho_2)d\x_{m_1+1,\ell-1}d\x_{\ell+1,m_1+n-2}\\
=&G_{N,0}(\varrho_1,\varrho_2)^{-1}\varrho_1(x_N)\prod_{k=m_1+n}^NG_{k,k-1}(x_k,x_{k-1})G_{m_1+n-1,\ell}(x_{m_1+n-1},x_\ell)\\
&\times G_{\ell,m_1}(x_\ell,x_{m_1}) \prod_{k=1}^{m_1} G_{k,k-1}(x_k,x_{k-1})\varrho_2(x_0).
\end{aligned}
\end{equation}
Recall the definition of forward and backward density in \eqref{020107-22}, we rewrite the two factors in \eqref{e.7191} that contain $x_\ell$ as
\begin{equation}\label{e.7192}
\begin{aligned}
&G_{m_1+n-1,\ell}(\cdot ,x_\ell)=\rho_{\rm per}(m_1+n-1,\cdot; \ell,\delta_{x_\ell})\int G_{m_1+n-1,\ell}(y ,x_\ell)dy,\\
&G_{\ell,m_1}(x_\ell,\cdot)=\tilde{\rho}_{\rm per}(\ell,\delta_{x_\ell};m_1,\cdot)\int G_{\ell,m_1}(x_\ell,y )dy.
\end{aligned}
\end{equation}
Further define the normalization constant
\begin{equation}\label{e.7193}
\begin{aligned}
&h_1(x_\ell)=\int
\varrho_1(x_N)\prod_{k=m_1+n}^NG_{k,k-1}(x_k,x_{k-1})\rho_{\rm per}(m_1+n-1, x_{m_1+n-1};\ell,\delta_{x_\ell})d\x_{m_1+n-1,N},\\
&h_2(x_\ell)=\int \tilde{\rho}_{\rm per}(\ell,\delta_{x_\ell};m_1,x_{m_1})\prod_{k=1}^{m_1} G_{k,k-1}(x_k,x_{k-1})\varrho_2(x_0) d\x_{0,m_1},
\end{aligned}
\end{equation}
and the densities
\begin{equation}\label{e.7194}
\begin{aligned}
  &p_1(\x_{m_1+n-1,N},x_\ell)=h_1(x_\ell)^{-1}\varrho_1(x_N)\prod_{k=m_1+n}^NG_{k,k-1}(x_k,x_{k-1})\\
  &
  \times \rho_{\rm
per}(m_1+n-1, x_{m_1+n-1};\ell,\delta_{x_\ell}),\\
&p_2(x_\ell,\x_{0,m_1})=h_2(x_\ell)^{-1}\tilde{\rho}_{\rm per}(\ell,\delta_{x_\ell};m_1,x_{m_1})\prod_{k=1}^{m_1} G_{k,k-1}(x_k,x_{k-1})\varrho_2(x_0).
\end{aligned}
\end{equation}

Using the above notations, one can rewrite \eqref{e.7191} as 
\[
\begin{aligned}
&\int \mu_N(\x;\varrho_1, \varrho_2)d\x_{m_1+1,\ell-1}d\x_{\ell+1,m_1+n-2}\\
&=p_1(\x_{m_1+n-1,N},x_\ell)p_2(x_\ell,\x_{0,m_1}) \mathbf{p}(x_\ell),
\end{aligned}
\]
where $\mathbf{p}(\cdot)$ is a density that takes the form
\[
\begin{aligned}
\mathbf{p}(x_\ell)=&G_{N,0}(\varrho_1,\varrho_2)^{-1}h_1(x_\ell)h_2(x_\ell)\\
&\times  \big(\int G_{m_1+n-1,\ell}(y ,x_\ell)dy\big)\big(\int G_{\ell,m_1}(x_\ell,y )dy\big).
\end{aligned}
\]
It is clear that $\mathbf{p}$ is the marginal density of $x_\ell$, since 
\[
\int \mu_N(\x;\varrho_1, \varrho_2)d\x_{0,\ell-1}d\x_{\ell+1,N}=\mathbf{p}(x_\ell).
\]
In this way, the expectation of the product is rewritten as
\begin{equation}\label{e.EFG}
\begin{aligned}
E[FG]=&\E_\varrho \EE \int  \E_\x[f\big(\eta_{1},\ldots,\eta_{m_1}\big)g\big(\eta_{m_1+n},\ldots,\eta_{m_2+n}\big)]\\
&\times p_1(\x_{m_1+n-1,N},x_\ell)p_2(x_\ell,\x_{0,m_1}) \mathbf{p}(x_\ell)d\x_{m_1+n-1,N} d\x_{0,m_1} dx_\ell.
\end{aligned}
\end{equation}

\emph{Step 2. Rewriting $E[F]E[G]$.} Let $\varrho_3,\varrho_4$ be sampled independently from $\pi_\infty$,  which are also independent from $\varrho_1,\varrho_2$ and the random environment. Define 
\begin{equation}\label{e.7195}
\begin{aligned}
  &p_3(\x_{m_1+n-1,N})\\
  &
  =\frac{\varrho_1(x_N)\prod_{k=m_1+n}^NG_{k,k-1}(x_k,x_{k-1})\rho_{\rm  per}(m_1+n-1 ,x_{m_1+n-1};\ell,\varrho_3)}{\int
\varrho_1(x_N')\prod_{k=m_1+n}^NG_{k,k-1}(x_k',x_{k-1}')\rho_{\rm per}(m_1+n-1, x_{m_1+n-1}';\ell,\varrho_3)d\x_{m_1+n-1,N}'},\\
&p_4(\x_{0,m_1})\\
&
=\frac{\tilde{\rho}_{\rm   per}(\ell,\varrho_4;m_1,x_{m_1})\prod_{k=1}^{m_1}
  G_{k,k-1}(x_k,x_{k-1})\varrho_2(x_0)}{\int \tilde{\rho}_{\rm per}(\ell,\varrho_4;m_1,x_{m_1}')\prod_{k=1}^{m_1} G_{k,k-1}(x_k',x_{k-1}')\varrho_2(x_0')d\x_{0,m_1}'}.
\end{aligned}
\end{equation}
In other words, in the expressions of $p_1,p_2$, we have replaced $\delta_{x_\ell}$ with $\varrho_3,\varrho_4$ to obtain $p_3,p_4$  respectively. Now it is straightforward to check that
\begin{equation}\label{e.EFEG}
\begin{aligned}
E[F]E[G]=\E_\varrho \EE &\int \E_\x[f\big(\eta_{1},\ldots,\eta_{m_1}\big)g\big(\eta_{m_1+n},\ldots,\eta_{m_2+n}\big)]\\
&\times p_3(\x_{m_1+n-1,N} )p_4( \x_{0,m_1}) \mathbf{p}(x_\ell) d\x_{m_1+n-1,N} d\x_{0,m_1}dx_\ell.
\end{aligned}
\end{equation}
In the above expression, the term $\mathbf{p}(x_\ell)$ actually plays no role since one can integrate it out and $\int \mathbf{p}(x_\ell)dx_\ell=1$ -- we kept it there to compare with the expression of $E[FG]$.

Combining \eqref{e.EFG} and \eqref{e.EFEG}, we have 
\begin{equation}\label{e.covFG1}
\begin{aligned}
\cov[F,G]=\E_\varrho \EE& \int  \E_\x[f\big(\eta_{1},\ldots,\eta_{m_1}\big)g\big(\eta_{m_1+n},\ldots,\eta_{m_2+n}\big)]\\
&\times [p_1(\x_{m_1+n-1,N},x_\ell)p_2(x_\ell,\x_{0,m_1})-p_3(\x_{m_1+n-1,N} )p_4( \x_{0,m_1})] \\
&\times \mathbf{p}(x_\ell)d\x_{m_1+n-1,N} d\x_{0,m_1} dx_\ell.
\end{aligned}
\end{equation}

\emph{Step 3. Approximation.} Now we decompose $\cov[F,G]=Err_1(n)+Err_2(n)$, with 
\[
\begin{aligned}
&Err_1(n)=\E_\varrho \EE \int \E_\x[f\big(\eta_{1},\ldots,\eta_{m_1}\big)g\big(\eta_{m_1+n},\ldots,\eta_{m_2+n}\big)]\\
&\times [p_1(\x_{m_1+n-1,N},x_\ell)-p_3(\x_{m_1+n-1,N})]p_2(x_\ell,\x_{0,m_1})  \mathbf{p}(x_\ell)d\x_{m_1+n-1,N} d\x_{0,m_1} dx_\ell,
\end{aligned}
\]
and
\[
\begin{aligned}
&Err_2(n)=\E_\varrho \EE \int  \E_\x[f\big(\eta_{1},\ldots,\eta_{m_1}\big)g\big(\eta_{m_1+n},\ldots,\eta_{m_2+n}\big)]\\
&\times p_3(\x_{m_1+n-1,N})[p_2(x_\ell,\x_{0,m_1})-  p_4( \x_{0,m_1})]\mathbf{p}(x_\ell)d\x_{m_1+n-1,N} d\x_{0,m_1} dx_\ell.
\end{aligned}
\]
It suffices to show that $|Err_i(n)|\leq Ce^{-\lambda n}\|F\|_{L^2}\|G\|_{L^2}$ for $i=1,2$. The two cases are handled in the same way, so we will only focus on $Err_1(n)$.

The rest of the proof is very similar to that of Proposition~\ref{prop010107-22}. First, from \eqref{e.7194} and \eqref{e.7195} we have 
\[
|p_1(\x_{m_1+n-1,N},x_\ell)-p_3(\x_{m_1+n-1,N})|\leq I_1+I_2,
\]
with 
\begin{align*}
&I_1=\frac{\varrho_1(x_N)\prod_{k=m_1+n}^NG_{k,k-1}(x_k,x_{k-1})
  }{\int
  \varrho_1(x_N')\prod_{k=m_1+n}^NG_{k,k-1}(x_k',x_{k-1}')d\x_{m_1+n-1,N}'\cdot
  \inf_y\rho_{\rm per}(m_1+n-1, y;\ell,\delta_{x_\ell})}\\
  &
    \times \sup_y|\rho_{\rm per}(m_1+n-1, y;\ell,\delta_{x_\ell})-\rho_{\rm
    per}(m_1+n-1, y;\ell,\varrho_3)|,
\end{align*}
\[
\begin{aligned}
I_2=&\frac{\varrho_1(x_N)\prod_{k=m_1+n}^NG_{k,k-1}(x_k,x_{k-1})}{\int \varrho_1(x_N')\prod_{k=m_1+n}^NG_{k,k-1}(x_k',x_{k-1}')d\x_{m_1+n-1,N}' }\\
&\times \frac{\sup_y\rho_{\rm per}(m_1+n-1,
  y;\ell,\varrho_3)}{\inf_y\rho_{\rm per}(m_1+n-1, y;\ell,\delta_{x_\ell})\inf_y\rho_{\rm
    per}(m_1+n-1, y;\ell,\varrho_3)}\\
&
\times \sup_y|\rho_{\rm per}(m_1+n-1,
  y;\ell,\delta_{x_\ell})-\rho_{\rm per}(m_1+n-1,
  y;\ell,\varrho_3)|.
\end{aligned}
\]
Thus, $Err_1(n)$ can be bounded from above by 
\[
\begin{aligned}
Err_1(n)\leq\E_\varrho \EE \int& \E_\x[|f\big(\eta_{1},\ldots,\eta_{m_1}\big)g\big(\eta_{m_1+n},\ldots,\eta_{m_2+n}\big)|]\\
&\times [I_1+I_2]p_2(x_\ell,\x_{0,m_1})  \mathbf{p}(x_\ell)d\x_{m_1+n-1,N} d\x_{0,m_1} dx_\ell=:J_1+J_2.
\end{aligned}
\]
Consider the term $J_1$. We first bound $I_1$ by 
\[
\begin{aligned}
I_1\leq &\frac{\varrho_1(x_N)\prod_{k=m_1+n}^NG_{k,k-1}(x_k,x_{k-1}) }{\int \varrho_1(x_N')\prod_{k=m_1+n}^NG_{k,k-1}(x_k',x_{k-1}') d\x_{m_1+n-1,N}'}\\
&\times  \frac{\sup_{y,x}|\rho_{\rm per}(m_1+n-1,
  y;\ell,\delta_{x})-\rho_{\rm per}(m_1+n-1, y;\ell,\varrho_3)|}{
  \inf_{y,x}\rho_{\rm per}(m_1+n-1, y;\ell,\delta_{x})}.
\end{aligned}
\]
Using \eqref{e.7194}, we bound $p_2(x_\ell,\x_{0,m_1})$ by 
\[
p_2(x_\ell,\x_{0,m_1})\leq \frac{\prod_{k=1}^{m_1}
  G_{k,k-1}(x_k,x_{k-1})\varrho_2(x_0)}{\int \prod_{k=1}^{m_1}
  G_{k,k-1}(x_k',x_{k-1}')\varrho_2(x_0')d\x_{0,m_1}'}\times
\frac{\sup_{x,y}
  \tilde{\rho}_{\rm
  per}(\ell,\delta_{x};m_1,y)}{\inf_{x,y}\tilde{\rho}_{\rm
  per}(\ell,\delta_{x};m_1,y)}.
\]
In this way we got rid of the dependence on $x_\ell$ in all other terms except for   $\mathbf{p}(x_\ell)$ which can be  integrated out. By the independence of the random environment in separate time intervals, we obtain $J_1\leq \prod_{i=1}^4 K_i$, 
with
\[
\begin{aligned}
K_1:= \E_\varrho\EE&\int \E_\x|f\big(\eta_{1},\ldots,\eta_{m_1}\big)|\frac{\prod_{k=1}^{m_1} G_{k,k-1}(x_k,x_{k-1})\varrho_2(x_0)}{\int \prod_{k=1}^{m_1} G_{k,k-1}(x_k',x_{k-1}')\varrho_2(x_0')d\x_{0,m_1}'} d\x_{0,m_1},\\
K_2:= \E_\varrho\EE& \int \E_\x|g\big(\eta_{m_1+n},\ldots,\eta_{m_2+n}\big)|\\
&\times \frac{\varrho_1(x_N)\prod_{k=m_1+n}^NG_{k,k-1}(x_k,x_{k-1}) }{\int \varrho_1(x_N')\prod_{k=m_1+n}^NG_{k,k-1}(x_k',x_{k-1}') d\x_{m_1+n-1,N}'} d\x_{m_1+n-1,N},\\
K_3:= \E_\varrho\EE & \frac{\sup_{y,x}|\rho_{\rm per}(m_1+n-1,
  y;\ell,\delta_{x})-\rho_{\rm per}(m_1+n-1, y;\ell,\varrho_3)|}{
  \inf_{y,x}\rho_{\rm per}(m_1+n-1, y;\ell,\delta_{x})},\\
K_4:=\E_\varrho\EE&\frac{\sup_{x,y}
  \tilde{\rho}_{\rm
  per}(\ell,\delta_{x};m_1,y)}{\inf_{x,y}\tilde{\rho}_{\rm
  per}(\ell,\delta_{x};m_1,y)}.\end{aligned}
\]
Applying Proposition~\ref{p.tk1}, we have $K_3\leq Ce^{-\lambda n}$
and $K_4\leq C$. For $K_1$, we can bound it from above by 
\[
\begin{aligned}
K_1\leq\E_\varrho\EE&\int\E_\x|f\big(\eta_{1},\ldots,\eta_{m_1}\big)|\\
&\times \frac{\varrho_4(x_{m_1})\prod_{k=1}^{m_1} G_{k,k-1}(x_k,x_{k-1})\varrho_2(x_0)}{\int \varrho_4(x_{m_1}') \prod_{k=1}^{m_1} G_{k,k-1}(x_k',x_{k-1}')\varrho_2(x_0')d\x_{0,m_1}'} d\x_{0,m_1} \times \frac{\sup_y \varrho_4(y)}{\inf_y \varrho_4(y)}.
\end{aligned}
\]
Applying the Cauchy-Schwarz inequality, we have 
\[
\begin{aligned}
K_1^2\leq C\E_\varrho\EE \bigg(\int&\E_\x|f\big(\eta_{1},\ldots,\eta_{m_1}\big)|\\
&\times \frac{\varrho_4(x_{m_1})\prod_{k=1}^{m_1} G_{k,k-1}(x_k,x_{k-1})\varrho_2(x_0)}{\int \varrho_4(x_{m_1}') \prod_{k=1}^{m_1} G_{k,k-1}(x_k',x_{k-1}')\varrho_2(x_0')d\x_{0,m_1}'} d\x_{0,m_1} \bigg)^2.
\end{aligned}
\]
The r.h.s. can be further bounded from above by 
\[
\begin{aligned}
 C\E_\varrho\EE \int&\E_\x|f\big(\eta_{1},\ldots,\eta_{m_1}\big)|^2\\
 &\times \frac{\varrho_4(x_{m_1})\prod_{k=1}^{m_1} G_{k,k-1}(x_k,x_{k-1})\varrho_2(x_0)}{\int \varrho_4(x_{m_1}') \prod_{k=1}^{m_1} G_{k,k-1}(x_k',x_{k-1}')\varrho_2(x_0')d\x_{0,m_1}'} d\x_{0,m_1} =C\|F\|_{L^2}^2.
\end{aligned}
\]
The same proof shows that $K_2 \leq C\|G\|_{L^2}$. So we have $J_1\leq Ce^{-\lambda n} \|F\|_{L^2}\|G\|_{L^2}$. The term $J_2$ is dealt with in the same way, and this combines to show that $Err_1(n)\leq Ce^{-\lambda n} \|F\|_{L^2}\|G\|_{L^2}$. Since $Err_2(n)$ is proved in exactly the same way, we complete the proof of the proposition.
\end{proof}

\section{Proof of Theorem~\ref{t.mainth} }
\label{s.clt1}
 


Recall the goal was to show that, as $T\to\infty$,
$\frac{w_T}{\sqrt{T}}$ converges to a nondegenerate Gaussian
distribution under the measure $\PP\otimes \hat{\Pb}_T$. 
As we have already observed in Section \ref{sec3.4}
the laws of 
$\frac{w_N}{\sqrt{N}}$  under the measure $\PP\otimes \hat{\Pb}_N$ weakly
converge to  $N(0,\sigma_{\mathrm{eff}}^2)$,   as  $N\to\infty$.
In this section we  show
how to extend the conclusion to the laws of $\frac{w_T}{\sqrt{T}}$,
under the measure $\PP\otimes \hat{\Pb}_T$, when $T\to\infty$ (not necessarily
taking integer values).

For a general $T>0$, define $N:=\lf T\rf $. 
Let $X_1(T)$ and $X_2(T)$ be  random variables with the same laws as $w_{N}$ under
$\PP\otimes \hat{\Pb}_N$ and  $w_{T}$ under $\PP\otimes \hat{\Pb}_T$, respectively. 
We have the following lemma, which is the last piece that is needed to complete the proof of Theorem~\ref{t.mainth}. 
\begin{lemma}
There exists $C>0$ such that for any $T\ge1$ one can find a coupling
 $\Big(X_1(T), X_2(T)\Big)$ such that
\begin{equation}
\label{X1X2}
\E \big|X_2(T)-X_1(T)\big|^2 \leq C.
\end{equation}
\end{lemma}

\begin{proof}
To lighten the notation we write $X_1$ and $X_2$, instead of $X_1(T)$ and $X_2(T)$. Recall that $\rho(T,\cdot)$ is the density of $w_T$ under $\hat{\Pb}_T$, and we have the  relation 
\begin{equation}\label{e.recur}
\rho(T,x)=\frac{\int Z_{T,N}(x,y)\rho(N,y)dy}{\int Z_{T,N}(x',y')\rho(N,y')dx'dy'}.
\end{equation}
From \eqref{e.recur}, we know that, given the value of $X_1$, if we sample $X_2$ from the density $\frac{ Z_{T,N}(\cdot,X_1)}{\int Z_{T,N}(x',X_1)dx'}$, then $X_2$ has the same law as $w_T$ under $\PP\otimes\hat{\Pb}_T$.
By the construction of $X_1,X_2$, we have for any $y\in\R$ that 
\[
\E[|X_2-X_1|^2|X_1=y]=\EE \frac{\int (x-y)^2Z_{T,N}(x,y)dx}{\int Z_{T,N}(x',y)dx'}.
\]
By the Cauchy-Schwarz inequality, we have 
\[
\begin{aligned}
&\E[|X_2-X_1|^2|X_1=y] \\
&\leq \int (x-y)^2 \sqrt{\EE Z_{T,N}(x,y)^2} dx \sqrt{ \EE (\int Z_{T,N}(x',y)dx')^{-2}}.
\end{aligned}
\]
For the first term on the r.h.s., we apply Lemma~\ref{l.mmZ} to derive (note that $T-N<1$)
\[
\int (x-y)^2 \sqrt{\EE Z_{T,N}(x,y)^2} dx \leq C.
\]
For the second term,  we have for each fixed $y$ that
\[
\int Z_{T,N}(x',y)dx'\stackrel{\text{law}}{=}v_{\1}(T-N,0),
\]
with $v_{\1}$ solving the SHE \eqref{e.she1a}, with the initial data
$v_{\1}(0,x)\equiv1$. Thus, applying \cite[Lemma B.7]{GK21}, we have $\EE (\int Z_{T,N}(x',y)dx')^{-2}\leq C$. The proof is
complete. 
\end{proof}

\section{Nondegeneracy of the diffusion constant}
\label{s.diff}

The goal of this section is to show that $\sigma_{\mathrm{eff}}^2\neq
0$, which is a nontrivial fact. In general, it needs not be true that $\sum_{j\in Z}r(j)>0$ when $r(\cdot)$ is the covariance function of a stationary sequence. We will need to make use of some specific structure of our model.

We first show a stability result on the approximation of the diffusion constant. Define 
\begin{equation}\label{e.defsigma1}
\sigma_N^2:=\frac{1}{N}\EE \int \E_{\x}\big(\sum_{k=1}^N \eta_k\big)^2 \mu_N(\x;{\rm m},\delta_0)  d\x_{0,N}=\frac{1}{N}\EE\hat{\E}_N \lf w_N\rf^2 ,
\end{equation}
and
\begin{equation}\label{e.defsigma2}
\tilde{\sigma}_N^2:=\frac{1}{N}\E_{\varrho}\E_{\tilde \varrho} \EE \int \E_{\x} \big(\sum_{k=1}^N \eta_k\big)^2  \mu_N(\x;\tilde\vrho,\vrho) d\x_{0,N}=\frac{1}{N}E(\sum_{k=1}^N \eta_k)^2.
\end{equation}
We have 
\begin{proposition}\label{p.varclose}
$\sigma_N^2-\tilde{\sigma}_N^2\to0$ as $N\to\infty$.
\end{proposition}

\begin{proof}
The proof is similar to that of Propositions~\ref{prop010107-22} and \ref{thm011207-22}, so we only sketch the argument here.

First, for some $k_N\to\infty$, yet to be determined, we decompose $\sigma_N^2=\sum_{\ell=1}^4 I_\ell$ and $\tilde{\sigma}_N^2=\sum_{\ell=1}^4 J_\ell$, with 
\[
\begin{aligned}
&I_\ell=\frac{1}{N}\sum_{(i,j)\in A_\ell}\EE \int_{\bT^{N+1}} \big(\E_{\x}\eta_i\eta_j\big) \mu_N(\x;{\rm m},\delta_0)  d\x_{0,N},\\
&J_\ell=\frac{1}{N}\sum_{(i,j)\in A_\ell}\E_{\varrho}\E_{\tilde \varrho} \EE \int_{\bT^{N+1}} \big(\E_{\x} \eta_i\eta_j \big)\mu_N(\x;\tilde\vrho,\vrho) d\x_{0,N},
\end{aligned}
\]
with 
\[
\begin{aligned}
A_1=B_1\times B_1, \quad\quad A_2=B_2\times B_2,\\
A_3=B_1\times B_2, \quad\quad A_4=B_2\times B_1,
\end{aligned}
\]
and $B_1=[k_N,N-k_N]$ and $B_2=[1,k_N-1]\cup [N-k_N+1,N]$. In the following, we will show that $I_\ell-J_\ell\to 0$, for each $\ell=1,\ldots,4$.

\emph{The case of $\ell=1$.} By following closely the proof of Proposition~\ref{prop010107-22} and with the help of Lemma~\ref{lm010107-22}, we derive that, for each $(i,j)\in A_1$,
\[
\begin{aligned}
&\left|\EE \int  \big(\E_{\x}\eta_i\eta_j\big) \mu_N(\x;{\rm m},\delta_0)  d\x_{0,N}-\E_{\varrho}\E_{\tilde \varrho} \EE \int \big(\E_{\x} \eta_i\eta_j \big)\mu_N(\x;\tilde\vrho,\vrho) d\x_{0,N}\right| \\
&\leq Ce^{-\lambda k_N},
\end{aligned}
\]
which implies that 
\[
|I_1-J_1| \leq CN^{-1}N^2 e^{-\lambda k_N}=CNe^{-\lambda k_N}.
\]

\emph{The case of $\ell=2$.} In this case, we directly apply Lemma~\ref{lm010107-22} to derive that 
\[
|I_2-J_2|\leq CN^{-1}k_N^2.
\]

\emph{The case of $\ell=3$ and $4$.} By symmetry, we only need to consider $\ell=3$. First, for $J_3$, we have 
\[
J_3=\frac{1}{N}\sum_{(i,j)\in A_3}E (\eta_i\eta_j) 
\]
By Proposition~\ref{thm011207-22}, we know that $|E(\eta_i\eta_j)|\leq Ce^{-\lambda |i-j|}$, and this implies that 
\[
|J_3|\leq CN^{-1} \sum_{i\in B_1,j\in B_2}e^{-\lambda|i-j|} \leq CN^{-1}k_N.
\]
It remains to study $I_3$, which we rewrite as
\[
\begin{aligned}
I_3=&\frac{1}{N}\sum_{(i,j)\in A_3}\left(\EE \int \big(\E_{\x}\eta_i\eta_j\big) \mu_N(\x;{\rm m},\delta_0)  d\x_{0,N}-\prod_{k=i,j}\EE \int \big(\E_{\x}\eta_k \big)\mu_N(\x;{\rm m},\delta_0)  d\x_{0,N}\right)\\
&+\frac{1}{N}\sum_{(i,j)\in A_3}\prod_{k=i,j}\EE \int ( \E_{\x}\eta_k) \mu_N(\x;{\rm m},\delta_0)  d\x_{0,N}=:I_{31}+I_{32}.
\end{aligned}
\]
For the term $I_{31}$, by following the proof for Proposition~\ref{thm011207-22}, one can show that 
\[
|I_{31}|\leq C N^{-1}\sum_{(i,j)\in A_3} e^{-\lambda |i-j|} \leq CN^{-1}k_N.
\]
For the other term, since $(i,j)\in A_3$, we have $i\in [k_N,N-k_N]$. By following the proof for Proposition~\ref{prop010107-22}, one can show that
\[
|\EE \int (\E_{\x}\eta_i) \mu_N(\x;{\rm m},\delta_0)d\x_{0,N}-\EE \int (\E_{\x}\eta_i) \mu_N(\x;\tilde{\varrho},\varrho)d\x_{0,N}|\leq Ce^{-\lambda k_N}.
\]
We have
$$
\EE \int (\E_{\x}\eta_i)
\mu_N(\x;\tilde{\varrho},\varrho)d\x_{0,N}=E\eta_i=0.
$$ 
We can apply Lemma~\ref{lm010107-22} again and conclude
\[
|I_{32}|\leq CN^{-1}\sum_{(i,j)\in A_3}e^{-\lambda k_N} \leq Ck_Ne^{-\lambda k_N}.
\]

To summarize, we have 
\[
|\sigma_N^2-\tilde{\sigma}_N^2| \leq \sum_{\ell=1}^4 |I_\ell-J_\ell|\leq C\big(Ne^{-\lambda k_N}+N^{-1}k_N^2+N^{-1}k_N+k_Ne^{-\lambda k_N}\big).
\]
Choosing $k_N=N^{\alpha}$ for any $\alpha\in (0,1/2)$, the proof is complete.
\end{proof}

Since
\[
\tilde{\sigma}_N^2=\frac{1}{N}E(\sum_{k=1}^N \eta_k)^2\to \sigma_{\mathrm{eff}}^2, \quad\quad \mbox{ as } N\to\infty,
\]
applying Proposition~\ref{p.varclose}, we derive that,  as $N\to\infty$,
\[
\sigma_N^2\to \sigma_{\mathrm{eff}}^2.
\]
The following proposition completes the proof of the nondegeneracy of $\sigma_{\mathrm{eff}}^2$:
\begin{proposition}\label{p.sigma}
We have 
\[
\liminf_{N\to\infty}\sigma^2_{N}\geq 1.
\] As a result, $\sigma_{\mathrm{eff}}^2\geq 1$.
\end{proposition}

\begin{proof}
By definition, we have 
\[
\sigma_N^2=\frac{1}{N}\EE\hat{\E}_N \lf w_N\rf^2.
\]
Since  $|w_N|\leq |\lfloor w_N\rfloor|+1$, we have via a triangle inequality that  
\[
\sqrt{\frac{1}{N} \EE\hat{\E}_N w_N^2} \leq \sqrt{\frac{1}{N}\EE\hat{\E}_N\lfloor w_N\rfloor^2}+\sqrt{\frac{1}{N}}.
\]
Sending $N\to\infty$ and applying Lemma~\ref{l.qvar} below, the proof is complete.
\end{proof}
%

\begin{lemma}\label{l.qvar}
For any $N\in \Z_+$, we have 
\begin{equation}
\EE\hat{\E}_Nw_N^2\geq N.
\end{equation}
\end{lemma}

\begin{proof}
Recall that $\rho(N,\cdot)$ is the density of $w_N$ under the quenched polymer measure $\hat{\Pb}_N$, we have 
\[
\begin{aligned}
\EE\hat{\E}_N w_N^2&=\EE \int x^2\rho(N,x)dx \\
&\geq \EE \Big[\int x^2\rho(N,x)dx-(\int x\rho(N,x)dx)^2\Big].
\end{aligned}
\]
Note that the last expression is just the average of the quenched variance. It remains to show 
\begin{equation}\label{e.varN}
\EE [\int x^2\rho(N,x)dx-(\int x\rho(N,x)dx)^2]=N,
\end{equation}
which is a standard folklore for the directed polymer when the random environment is statistically invariant under shear transformations. We sketch the argument below.

For any $\theta\in\R$, consider the solution to SHE 
\begin{equation}\label{e.she1}
\begin{aligned}
&\partial_tZ_\theta(t,x)=\frac12\Delta Z_\theta(t,x)+\beta \xi(t,x)Z_\theta(t,x), \quad\quad t>0,\\
&Z_\theta(0,x)=e^{\theta x},
\end{aligned}
\end{equation}
 then we know that, for fixed $N$, 
\[
\int x^2\rho(N,x)dx-(\int x\rho(N,x)dx)^2\,\stackrel{\text{law}}{=}\,\partial_\theta^2 \log Z_\theta (N,0)\big|_{\theta=0}.
\]
\blue{For a proof of this fact, we refer to e.g. \cite[Eq. (2.15)]{gutams}.}
Since $\xi$ is a Gaussian process that is white in time and stationary in space, we have 
\begin{equation}\label{e.shear}
\{\xi(t,x)\}_{t,x}\stackrel{\text{law}}{=}\{\xi(t,x+\theta t)\}_{t,x},
\end{equation}
and this implies 
\begin{equation} 
\{Z_\theta(t,x)\}_{t,x}\stackrel{\text{law}}{=}\{Z_0(t,x+\theta t)e^{\theta x+\frac12\theta^2 t}\}_{t,x},
\end{equation}
which comes from the fact that $Z_0(t,x+\theta t)e^{\theta x+\frac12\theta^2 t}$ solves \eqref{e.she1} with $\{\xi(t,x)\}$ replaced by $\{\xi(t,x+\theta t)\}$.
Therefore, we have
\[
\begin{aligned}
\EE \log Z_\theta(N,0)&=\frac12\theta^2 N+\EE \log Z_0(N,\theta N)\\
&=\frac12\theta^2 N+\EE \log Z_0(N,0),
\end{aligned}
\]
where the last step comes from the fact that $Z_0(N,\cdot)$ is a stationary random field. Therefore, we have 
\[
\EE \partial_\theta^2 \log Z_\theta (N,0)\big|_{\theta=0}=N.
\]
 The proof is complete.
\end{proof}

\section{Further discussion}
\label{s.dis}

We list two  problems here.

\subsubsection*{Quenched behavior}
Theorem~\ref{t.mainth} concerns the behavior of $w_T$ under the annealed measure $\PP\otimes \hat{\Pb}_T$, and our approach does not give information on the quenched behavior.  For the problem on the whole line with no periodic structure,  the annealed and quenched behaviors are quite different, due to the localization phenomenon. It would be interesting to study the quenched asymptotics of $w_T$ in our setting.

\subsubsection*{Relation between two diffusion constants} Recall that $Z_T$ is the partition function formally defined in \eqref{e.defZ}. It was shown in \cite{GK21}, under the same assumption as here, that the free energy $\log Z_T$ satisfies a central limit theorem: there exists $\gamma,\Sigma>0$ such that under $\PP$, 
\[
\frac{\log Z_T+\gamma T}{\sqrt{T}}\Rightarrow N(0,\Sigma), \quad\quad \mbox{ as } T\to\infty.
\]
Different expressions of $\Sigma$ were derived, see \cite[Eq. (5.58)]{GK21} which involves the solution to an abstract cell problem and \cite[Eq. (2.10)]{dgk} which takes the form of an average of a Brownian bridge functional. A surprising relation between $\Sigma$ and $\sigma_{\mathrm{eff}}^2$ was derived by Brunet through the replica method, see  \cite[Eq. (20)]{brunet}. It is unclear at all why they should be related, and we believe it is an important problem to unravel  the connection here.

\appendix

\section{Proof of Lemma \ref{l.mmZ}}
\label{appA}

We start with the following.
\begin{lemma}
\label{lmA1}
For   each $t>s$ and $y\in\R$, the process  $\{Z_{t,s}(x,y)/q_{t-s}(x-y)\}_{x\in\R}$ is stationary.
\end{lemma}
\proof 
The argument is  rather standard, so we only sketch the proof. 
We consider first the case when $\xi_R(t,x)$ is a $1$-periodic Gaussian  noise
that is white in time and colored in space, with the covariance
function $R(\cdot)\in C^\infty(\bT)$. Recall that $q_t(x)=(2\pi
t)^{-1/2}\exp(-\frac{x^2}{2t})$ denotes the standard heat kernel.  The propagator of equation
\eqref{e.she}, corresponding to this noise, shall be
denoted by $Z_{t,s}^{(R)}(x,y)$ and is given by the formula
\begin{align*}
Z_{t,s}^{(R)}(x,y)= q_{t-s}(x-y) \E \left[\exp\left\{\beta\int_s^t\xi_R( \sigma,B^{y,x}_{s,t}(\sigma)) d \sigma
    -\frac12\beta^2 R(0)(t-s)\right\} \right],
\end{align*}
where $(B^{y,x}_{s,t}(\sigma))_{s\le \sigma\le t}$ is the Brownian bridge between $(s,y)$
and $(t,x)$.  It is clear from the above formula that
$\{Z_{t,s}^{(R)}(x,y)/q_{t-s}(x-y)\}_{x\in\R}$ is stationary, using the fact that $\xi_R$ also satisfies the relation \eqref{e.shear}.  The
conclusion can be extended to the case of the   Gaussian  space-time
white noise by approximation of $\delta(x-y)$ by a   sequence
of $R_n(\cdot)\in C^\infty(\bT)$ as $n\to\infty$. 
\qed

\subsection*{The end of the proof of  Lemma \ref{l.mmZ}}
Using Lemma \ref{lmA1} and the definition of the propagator of the SHE
on a torus, see \eqref{Gts}, we can write
\[
\EE Z_{t,0}(x,0)^p=\big(\frac{q_t(x)}{q_t(0)}\big)^p\EE
Z_{t,0}(0,0)^p\le \big(\frac{q_t(x)}{q_t(0)}\big)^p\EE G_{t,0}(0,0)^p.
\]
Applying \cite[Lemma B.1]{GK21}, we complete the proof.


\begin{thebibliography}{99}

\bibitem{bil} P. Billingsley, {\em Convergence of Probability
    Measures, 2nd edt}, Wiley and Sons, 1999.
    
    
    
    \bibitem{brunet} E. Brunet, {\em Fluctuations of the winding number of a directed polymer in a random medium},  Physical Review E 68.4 (2003): 041101.
    
    \bibitem{cc}
    G. Cannizzaro, K. Chouk, {\em Multidimensional SDEs with singular drift and universal construction of the polymer measure with white noise potential}, The Annals of Probability 46.3 (2018): 1710-1763.
  

\bibitem{otto}
G. Chatzigeorgiou, P. Morfe, F. Otto, L. Wang, 
{\em The Gaussian free-field as a stream function: asymptotics of effective diffusivity in infra-red cut-off}, arXiv preprint arXiv:2212.14244. 



\bibitem{corwin}
I. Corwin, A. Hammond, {\em KPZ line ensemble}, Probability Theory and Related Fields 166.1 (2016): 67-185.

\bibitem{dr}
F. Delarue, D. Roland,  {\em Rough paths and 1d SDE with a time dependent distributional drift: application to polymers},  Probability Theory and Related Fields 165.1 (2016): 1-63.



\bibitem{dgk}
A. Dunlap, Y. Gu, T. Komorowski, {\em Fluctuation exponents of the KPZ
  equation on a large torus},  arXiv preprint arXiv:2111.03650.



\bibitem{GK21}
Y. Gu, T. Komorowski, {\em KPZ on torus: Gaussian fluctuation}, arXiv preprint arXiv:2104.13540 (2021).


\bibitem{gutams}
Y. Gu, T. Komorowski, {\em Another look at the Bal\'asz-Quastel-Sepp\"al\"ainen theorem}, arXiv preprint arXiv:2203.03733 (2022).


\bibitem{gp}
M. Gubinelli, N. Perkowski, {\em KPZ reloaded}, Communications in Mathematical Physics 349.1 (2017): 165-269.

\bibitem{gip}
M. Gubinelli, P. Imkeller, N. Perkowski, {\em Paracontrolled distributions and singular PDEs}, Forum of Mathematics, Pi. Vol. 3. Cambridge University Press, 2015.

\bibitem{hairer1}
M. Hairer, {\em Solving the KPZ equation}, Annals of mathematics (2013): 559-664.


\bibitem{hairer2}
M. Hairer, {\em A theory of regularity structures}, Inventiones mathematicae 198.2 (2014): 269-504.

\bibitem{davar}
D. Khoshnevisan, {\em Analysis of stochastic partial differential equations}, Vol. 119. American Mathematical Soc., 2014.

\bibitem{rosati}
T. Rosati, {\em Synchronization for KPZ}, arXiv preprint arXiv:1907.06278 (2019).

\bibitem{sinai}
Y. Sinai, {\em Two results concerning asymptotic behavior of solutions of the Burgers equation with force}, Journal of statistical physics 64.1 (1991): 1-12.

\end{thebibliography}
\end{document}